\documentclass[12pt]{amsart}
\usepackage{amssymb,latexsym,amscd,verbatim,color}
\usepackage[T1]{fontenc}
\usepackage{mathrsfs}
\usepackage{xy}
\xyoption{all}
\usepackage[colorlinks,linkcolor=blue,citecolor=blue,urlcolor=black]{hyperref}
\usepackage{ulem}
\definecolor{labelkey}{rgb}{1,0,0}

\swapnumbers
\numberwithin{equation}{section}
\newtheorem{thm}{Theorem}[subsection]
\newtheorem{propose}[thm]{Proposition}
\newtheorem{lemma}[thm]{Lemma}

\theoremstyle{definition}
\newtheorem{defn}[thm]{Definition}
\newtheorem{remark}[thm]{Remark}
\newtheorem{remarks}[thm]{Remarks}
\newtheorem{example}[thm]{Example}
\newtheorem{examples}[thm]{Examples}

\newcommand{\M}{\mathcal{M}_{1}}

\newcommand{\MQ}{\mathcal{M}_{1,\mathbb Q}}
\newcommand{\Ma}{\mathcal{M}_{1}^{\mathrm a}}
\newcommand{\Mal}{\mathcal{M}_{1}^{\mathrm a,l}}

\newcommand{\Lie}{\mathrm{Lie}}
\newcommand{\uLie}{\underline{\mathrm{Lie}} }
\newcommand{\bLie}{\mathbf{Lie}}
\newcommand{\spec}{\mathrm{Spec}}
\newcommand{\eff}{\mathrm{eff}}
\newcommand{\dr}{\mathrm{dR}} \newcommand{\ab}{\mathrm{ab}}
\newcommand{\cris}{\mathrm{cris}}
\newcommand{\og}{\mathrm{Og}}
\newcommand{\bog}{\mathrm{BOg}}

\newcommand{\V}{\mathbb{V}}

\newcommand{\p}{\mathfrak{p}}

\renewcommand{\d}{{\text{\LARGE $\cdot $}}}
\renewcommand{\hat}{\widehat}
\newcommand{\Spec}{\operatorname{Spec}} 
\newcommand{\Hom}{\operatorname{Hom}}      
\newcommand{\Ext}{\operatorname{Ext}}      
\newcommand{\DM}{\operatorname{DM}}          

\newcommand{\End}{\operatorname{End}}      

\newcommand{\C}{\mathbb{C}}     
\newcommand{\F}{\mathbb{F}}
\newcommand{\Q}{\mathbb{Q}}     
\newcommand{\Z}{\mathbb{Z}}     
\renewcommand{\L}{\mathbb{L}}
\newcommand{\N}{\mathbb{N}}
\newcommand{\G}{\mathbb{G}}     

\renewcommand{\ker}{\operatorname{Ker}}  
\newcommand{\gr}{\operatorname{gr}}        
\newcommand{\Pic}{\operatorname{Pic}}     
\newcommand{\LAlb}{\operatorname{LAlb}}     

\newcommand{\LA}[1]{\mbox{${\rm L}_{#1}{\rm Alb}$}}

\newcommand{\longby}[1]{\stackrel{#1}{\longrightarrow}}

\renewcommand{\tilde}{\widetilde}
\newcommand{\df}{\mbox{\,${:=}$}\,}
\newcommand{\ie}{{\it i.e.}, }
\newcommand{\cf}{{\it cf. }}
\newcommand{\eg}{{\it e.g. }}

\newcommand{\loccit}{{\it loc. cit. }}

\newcommand{\gm}{{\rm gm}}

\renewcommand{\bar}{\overline}

\renewcommand{\lim}{\varprojlim}
\newcommand{\colim}{2\!\!\text{ -}\!\!\varinjlim}

\newcommand{\boxtensor}{\def\boxtimesten{\Box\kern-7.59pt\raise1.2pt
\hbox{$\times$} }}                                  

\newcounter{elno}                   

\newcommand{\cO}{\mathcal{O}}
\newcommand{\cP}{\mathcal{P}}

\newcommand{\cV}{\mathcal{V}}

\renewcommand{\phi}{\varphi}
\renewcommand{\epsilon}{\varepsilon}

\begin{document}
\title{Ogus realization of $1$-motives}
\author{F. Andreatta, L. Barbieri-Viale}
\address{Dipartimento di Matematica ``F. Enriques'', Universit{\`a} degli Studi di Milano\\ Via C. Saldini, 50\\ I-20133 Milano\\ Italy}\email{Fabrizio.Andreatta@unimi.it} \email{Luca.Barbieri-Viale@unimi.it}
\author{A. Bertapelle}
\address{Dipartimento di Matematica Pura ed Applicata, Universit\`a   degli Studi di Padova\\ vis Trieste, 63\\Padova -- I-35121\\ Italy} \email{alessandra.bertapelle@unipd.it}

\begin{abstract}
After introducing the Ogus realization of $1$-motives we prove that it is a fully faithful functor.
More precisely, following a framework introduced by Ogus, considering an enriched structure on the de Rham realization of $1$-motives over a number field,  we show that it yields a full functor by making use of an algebraicity theorem of Bost.
\end{abstract}

\maketitle

\section*{Introduction}
The Ogus realization of motives over a number field is considered as an analogue of the Hodge realization over the complex numbers and the $\ell$-adic realization over fields which are finitely generated over the prime field. The fullness of these realizations, along with the semi-simplicity of the essential image of pure motives, is a longstanding conjecture which implies the Grothendieck standard conjectures (\eg see \cite[\S 7.1]{An}).

The named conjecture on fullness is actually a theorem if we restrict to the category of abelian varieties up to isogenies regarded as the semi-simple abelian $\Q$-linear category of {\it pure} $1$-motives (\eg see \cite[Prop. 4.3.4.1 \& Thm. 7.1.7.5]{An}). A natural task is then to extend this theorem to {\it mixed}  $1$-motives up to isogenies. For the Hodge realization this goes through Deligne's result on the algebraicity of the effective mixed polarizable Hodge structures of level $\leq 1$ (see \cite[\S 10.1.3]{De}).  For the $\ell$-adic realization, the fullness follows from the Tate conjectures for abelian varieties (proven by Faltings) and the fullness for $1$-motives is proven by Jannsen (see \cite[\S 4]{JA}).\\

The main task of this paper is to show that there is a suitable version of Ogus realization for $1$-motives such that the fullness can be achieved: this is Theorem \ref{thm.tog} below.  This result for abelian varieties relies on a theorem of Bost as explained by Andr\'e (see \cite[\S 7.4.2]{An}).  For pure $0$-motives is our Lemma~\ref{pro.tog1}. 
However, in the mixed case, our theorem doesn't follow directly from Bost's theorem nor Andr\'e's arguments (see  Example \ref{ex.fund}).\\

For a number field $K$, recall that the Ogus category $\mathbf{Og}(K)$ is the $\Q$-linear abelian category whose objects are finite dimensional $K$-vector spaces $V$ such that the $v$-adic completion $V_{v}$ is endowed, for almost every unramified place $v$ of $K$, of a bijective semilinear endomorphism $F_{v}$. Actually, we here introduce an enriched version of the Ogus category denoted by $\mathbf{FOg}(K)$ whose objects $\cV\in  \mathbf{Og}(K)$ are endowed with an increasing finite exhaustive weight filtration $W_{\d}\cV$ (see Section 1 for details). We provide a realization functor
\[T_{\og}\colon \MQ \to \mathbf{FOg}(K)\]
where $\MQ$ is the abelian $\Q$-linear category of $1$-motives up to isogenies (see Proposition \ref{OGreal}). With the notation adopted below, this $T_{\og}({\sf M}_{K})$ of a $1$-motive ${\sf M}_{K}$ (see Definition \ref{TOg}), is given by $V\df T_{\dr}({\sf M}_{K})$ the de Rham realization (see Definition \ref{Tdr}), as a $K$-vector space, so that $V_{v}\simeq T_{\dr} ({\sf M}_{\cO_{K_{v}}} )\otimes_{\cO_{K_{v}}} K_{v}$ and $F_{v}\df (\Phi_{v}\otimes {\rm id})^{-1}$ where the $\sigma_{v}^{-1}$-semilinear endomorphism $\Phi_{v}$ (Verschiebung) on $T_{\dr}({\sf M}_{\cO_{K_{v}}})$ is obtained via the canonical isomorphism $T_{\dr}({\sf M}_{\cO_{K_{v}}})\simeq T_\cris({\sf M}_{k_{v}} )$ given by the comparison with the crystalline realization of the special fiber ${\sf M}_{k_{v}}$ for every unramified place $v$ of good reduction for ${\sf M}_K$ (see \cite[\S 4, Cor. 4.2.1]{ABV}).  

In the proof of the fullness of $T_{\og}$ for arbitrary $1$-motives, with nontrivial discrete part, the main ingredient is the fact that the decomposition of $V_{v}$ as a sum of pure $F$-$K_{v}$-isocrystals is realized geometrically via the $p$-adic logarithm (see Section 2, in particular Lemma \ref{d.lbd}, and  the key Lemma \ref{section}).

Note that the essential image of $T_{\og}$ is contained in the category $\mathbf{FOg}(K)_{(1)}$ of effective objects of level $\leq 1$ (see Definition \ref{fog1}). After Theorem~\ref{thm.tog} it is clear that the Ogus realization of pure $1$-motives in $\MQ$ is semi-simple.
However, the characterization of the essential image is an open question (unless we are in the case of Artin motives, \cf \cite{KW}).\\

Further directions of investigation are related to Voevodsky motives. Recall that the derived category of $1$-motives $D^b(\MQ)$ can be regarded as a reflective triangulated full subcategory of effective Voevodsky motives $\DM_{\gm}^{\eff}$ (see \cite[Thm. 6.2.1 \& Cor. 6.2.2]{BVK}). In fact, there is a functor $\LAlb^\Q\colon \DM_\gm^\eff\to D^b(\MQ)$ which is a left adjoint to the inclusion. Thus, associated to any algebraic $K$-scheme $X$, we get a complex of $1$-motives $\LAlb^\Q (X)\in D^b(\MQ)$. 
Whence, by taking $i$-th homology, we get $\LA{i}^\Q(X)_{K}\in \MQ$ which, most likely, is the geometric avatar of level $\leq 1$ Ogus $i$-th homology of $X$, \ie the Ogus realization $T_{\og}(\LA{i}^\Q(X))\in\mathbf{FOg}(K)_{(1)}$ is the largest quotient of level $\leq 1$ of the $i$-th Ogus homology of $X$, according with the framework of Deligne's conjecture (see \cite[\S 14]{BVK} and compare with \cite[Conj. C]{ABV}). 
This is actually the case for the underlying $K$-vector spaces by the corresponding result for the mixed realization but Ogus realization (see \cite[Thm. 16.3.1]{BVK}).

\subsubsection*{Notation}
We here denote by $K$ a number field. For any finite place $v$ of $K$, \ie any prime ideal of  the ring of integers $\cO_{K}$ of $K$, we let $K_{v}$ be the completion of $K$ with respect to the valuation $v$ and $\cO_{K_{v}}$ its ring of integers.  Let $\p_{v}$ be the unique maximal ideal of $\cO_{K_{v}}$, $k_{v}$ the residue field of $\cO_{K_{v}}$, $p_{v}$ its characteristic and $n_{v}\df [k_{v}:\F_{p_{v}}]$. If  $v$ is unramified, we let $\sigma_{v}$ be the canonical Frobenius map on $K_{v}$ and $\sigma_{v}$ will also denote the Frobenius maps on $\cO_{K_{v}}$ and on $k_{v}$.

\section{Ogus categories}

\subsection{The plain Ogus category $\mathbf{Og}(K)$} As in \cite[\S 7.1.5]{An} let $\mathbf{Og}(K)$ be the $\Q$-linear abelian category whose objects are finite dimensional $K$-vector spaces $V$ such that the $v$-adic completion $V_{v}=V\otimes_{K} K_{v}$ is equipped for almost every unramified place $v$ of $K$, of a bijective semilinear endomorphism $F_{v}$, \ie  for any $\alpha\in K_{v},x,y\in V_{v}$ we have $F_{v}(\alpha x+y)=\sigma_{v}(\alpha)F_{v}(x)+F_{v}(y)$.
Morphisms in $\mathbf{Og}(K)$ are $K$-linear maps compatible with the $F_{v}$'s for almost all $v$.

To give a more precise definition of the above category we have to present it as a $2$-colimit  category.
Let $\cP$ denote the set of unramified places of $K$. For any cofinite subset  $\cP^{\prime}$ of $\cP$, let $\mathcal C_{\cP^{\prime}}$ denote the $\Q$-linear category whose objects are of the type $\left(V, (V_{v}, F_{v})_{v\in \cP^{\prime}}, (g_{v})_{v\in \cP^{\prime}}\right)$ where $V$ is a finite dimensional
$K$-vector space, $V_{v}$ are finite dimensional $K_{v}$-vector spaces, $F_{v}$ is a bijective $\sigma_{v}$-semilinear endomorphism on $V_{v}$, and  $g_{v}\colon V\otimes_{K} K_{v}\to V_{v}$ is an isomorphism of $K_{v}$-vector spaces. Morphisms are morphisms of vector spaces (on $K$ and on $K_{v}$) which respect the given structures ($F_{v}$ and $g_{v}$). Clearly,
for any cofinite subset $\cP^{\prime\prime}$ of $\cP^{\prime}$,  we have a canonical restriction functor $i_{\cP^{\prime},\cP^{\prime\prime}}\colon \mathcal C_{\cP^{\prime}}\to \mathcal C_{\cP^{\prime\prime}}$
which simply forgets the data for $v\in \cP^{\prime}\smallsetminus \cP^{\prime\prime}$.
\begin{defn}\label{OgDef}
Let
\[\mathbf{Og}(K):= \colim_{\cP^{\prime}\subseteq \cP}\ \mathcal{C}_{\cP^{\prime}}
\]
be the 2-colimit category. If $\cV_{1}\in \mathcal C_{\cP_{1}}$ and $\cV_2\in \mathcal C_{\cP_{2}}$ then 
\[   \mathrm{Hom}_{\mathbf{Og}(K)}(\cV_{1},\cV_{2})=\varinjlim_{\cP_3\subset \cP_{1}\cap\cP_{2}} \mathrm{Hom}_{\mathcal C_{\cP_3}}(i_{\cP_{1},\cP_3} (\cV_{1}),i_{\cP_{2},\cP_3} (\cV_{2})). \]

\begin{remark}\label{rem.pn}
Given a positive integer $n$ we denote by $\cP_{n}\subset \cP$ the subset of  places  $v$ such that $n$ is invertible in $\cO_{K_{v}}$. Since any cofinite subset $\cP'$ in $\cP$ contains $\cP_{n}$, for any $n$ divisible by $p_v$ with $v\in \cP\smallsetminus\cP^{\prime}$, we may equivalently get  $\mathbf{Og}(K)$ as  $2\!\text{ -}\!\varinjlim_{n} \mathcal{C}_{\cP_{n} }$.
\end{remark}

For a given  $\cV\df \left(V, (V_{v}, F_{v})_{v\in \cP^{\prime}}, (g_{v})_{v\in \cP^{\prime}}\right)$ in $ \mathcal C_{\cP^{\prime}}$ and an integer $n\in \Z$ we define the twist $\cV (n) \df \left(V, (V_{v}, p^{-n}_{v}F_{v})_{v\in \cP^{\prime}}, (g_{v})_{v\in \cP^{\prime}}\right)$.
\end{defn}

\subsection{$F$-$K_{v}$-isocrystals}

Recall that a {\it pure} $F$-$K_{v}$-isocrystal  $(H,\psi)$ of integral weight $i$ relative to $k_{v}$ in the sense of \cite[II, \S 2.0]{Ch} is a $K_{v}$-vector space  $H$  together with a $K_{v}$-linear endomorphism $\psi$ such that the eigenvalues of $\psi$ are  Weil numbers of weight $i$ relative to $k_{v}$.  Namely, they are algebraic numbers such that they and  all their conjugates have archimedean absolute value equal to $p_{v}^{in_{v}/2}$. In particular, $\psi$ is an isomorphism. The trivial space $H=0$ is assumed to be a pure $F$-$K_{v}$-isocrystal of any weight.
We say that $(H,\psi)$ is a {\it mixed} $F$-$K_{v}$-isocrystal with integral weights relative to $k_{v}$ if it admits a finite increasing  $K_{v}$-filtration $$0=W_mH\subseteq
\cdots \subseteq W_iH\subseteq W_{i+1}H\subseteq \cdots \subseteq W_nH=H$$  respected by $\psi$ such that $\gr_{i}^WH_{\d}\df W_{i}H/W_{i-1}H$ is pure of weight $i$ for any $ m< i\leq n$.
A morphism  $(H,\psi)\to (H^{\prime},\psi^{\prime}) $ of $F$-$K_{v}$-isocrystals is a homomorphism of $K_{v}$-vector spaces $f\colon H\to H^{\prime}$ such that $f\circ\psi^{\prime}=\psi\circ f$.

\begin{lemma}\label{wn} Let $H$ be  a $K_{v}$-vector space and $\psi\colon H\to H$ an  endomorphism. The following are equivalent:
\begin{itemize}
\item[(i)] $(H,\psi)$ is a mixed $F$-$K_{v}$-isocrystal.
\item[(i')] $(H,\psi)$ admits a {\it unique} finite increasing  $K_{v}$-filtration $0=W_{m}H\subseteq   \cdots \subseteq W_{n}H=H$  respected by $\psi$ such that $\gr_{i}^W H\df W_{i}H/W_{i-1}H$ is pure of weight $i$ for each $ m< i\leq n$.
\item[(ii)] All eigenvalues of $\psi$ are Weil numbers of integral weight relative to $k_{v}$.  
\item[(iii)] $H$ has a unique decomposition $\oplus_{i=m+1}^{n} H_i$ by $\psi$-stable vector subspaces so that  $(H_i, \psi_{|H_i})$ is a pure $F$-$K_{v}$-isocrystal  of  weight $i$.
\end{itemize}
\end{lemma}

\begin{proof}
$(i)\Rightarrow (ii)$. Let $W_{\d}H$ be the filtration of $H$ and let $\bar\psi_{i}$ denote the endomorphism of $\gr_{i}^WH$ induced by $\psi$. The eigenvalues of $\bar \psi_{i}$ are Weil numbers of integral weight by hypothesis. Since the characteristic polynomial of $\psi$ is the product of the characteristic polynomials of $\psi_{|W_{n-1}H}$ and  $\bar\psi_{n}$, one   proves recursively that all eigenvalues of $\psi$ are Weil numbers (of integral weight). \\ 
$(ii)\Rightarrow (iii)$. If all eigenvalues of $\psi$ are in $K_{v}$, then  $H$ is the direct sum of its  generalized eigenspaces. Let  $H_i$ be the direct sum of generalized eigenspaces associated to eigenvalues of weight $i$. Then $(H_i,\psi_{|H_i})$ is  a pure $F$-$K_{v}$-isocrystal of weight $i$ and  $H=\oplus_{i} H_i$.
In the general case, let $L/K_{v}$ be a finite Galois extension containing all the eigenvalues of $\psi$ and let $k_{L}$ be its residue field.  Observe that $(H\otimes_{K_{v}} L,\psi\otimes {\rm id})$ is not an $F$-$L$-isocrystal in general. Indeed if an eigenvalue $\alpha$ of $\psi$ has weight $i$ (relative to $k_{v}$), \ie $|\alpha|=p_{v}^{n_{v}i/2}$ then $|\alpha|=p_{v}^{n_{v}ri/{(2r)}}$, with $p^{n_{v}r}_{v}=|k_{L}|$, would have weight $i/r$ relative to $k_{L}$ and $i/r$ might not be an integer. 
However, the  decomposition result works the same if we consider rational weights. Hence $H\otimes_{K_{v}} L =\oplus_{i} (H\otimes_{K_{v}} L)_{i/r}$ where the $L$-linear subspace $(H\otimes_{K_{v}} L)_{i/r}$ is the direct sum of generalized eigenspaces associated to eigenvalues of modulus $p_{v}^{in_{v}/2}= p_{v}^{in_{v}r/{(2r)}}$. 
Since all conjugates of a Weil number of weight $i$ have the same weight, the conjugate of any eigenvalue $\alpha$ of $\psi$ has the same weight. Hence the  action of $\mathrm{Gal}(L/K_{v})$ on $H\otimes_{K_{v}} L$  respects the decomposition and $ (H\otimes_{K_{v}} L)_{i/r}$   descends to a $K_{v}$-linear subspace $H_i$ of $H$ which is  a pure $F$-$K_{v}$-isocrystal of weight $i$.

Assume now  $H=\oplus_{i=m+1}^{n} H^{\prime}_i$ is another decomposition as in (iii). In order to prove that $H^{\prime }_i=H_i$ we may assume $L=K_{v}$. Now $H_i$ is sum of the generalized eigenspaces  associated to eigenvalues of weight $i$. Since the eigenvalues of $\psi_{| H^{\prime}_i}$ have weight $i$, it is $H^{\prime}_i\subseteq H_i$. One concludes then by dimension reason.
\\$(iii)\Rightarrow (i)$. Let $\oplus_{i=m+1}^{n} H_{i}$ be a decomposition as in (iii) and set $W_mH=0$ and $W_{i}H=\oplus_{j=m+1}^{i}H_{j}$ for $m<i\leq n$. This filtration makes $H$ a mixed  $F$-$K_{v}$-isocrystal. 
Let $W_{\d}^\prime H$ be another filtration making $(H,\psi)$ a $F$-$K_{v}$-isocrystal. One proves recursively that the eigenvalues of $\psi_{|W^{\prime}_{i}H}$ have weights $\leq i$. In particular, the image of $W_{i}^\prime H$ in   $ \gr_{i+1}^WH_{\d} $ is trivial. Hence $W_{i}^\prime H\subseteq W_{i}H$. The reverse inclusion is proved analogously. Hence $W_{\d}H=W_{\d}^\prime H$. 
\end{proof}

It follows from the lemma above that any morphism of  $F$-$K_{v}$-isocrystals respects the decompositions in pure $F$-$K_{v}$-isocrystals, thus the filtrations.

\subsection{The enriched Ogus category  $\mathbf{FOg}(K)$}\label{sec:mixedOgus}

We say that an object $\cV$ of $\mathbf{Og}(K)$ is pure of weight $i$ if for almost all unramified places  $v$ the $K_{v}$-vector space $V_{v}$ and the $K_{v}$-linear
operator $F_{v}^{n_{v}}$ on $V_{v}$ define a pure $F$-$K_{v}$-isocrystal of weight $i$ relative to $k_{v}$.
\begin{defn}
Let  $\mathbf{FOg}(K)$ be the category whose objects are objects  $\cV$ in $\mathbf{Og}(K)$ endowed with an increasing finite exhaustive filtration
 \[0= W_{m}\cV\subseteq
\cdots \subseteq W_{i}\cV\subseteq W_{i+1}\cV\subseteq \cdots \subseteq W_{n}\cV=\cV \] in
 $\mathbf{Og}(K)$ such that for every $i>m$ the graded $\gr_{i}^W\cV\df W_{i}\cV/W_{i-1}\cV$ is pure of weight $i$. 
In particular, $(V_{v}, F_{v}^{n_{v}})$ is a mixed $F$-$K_{v}$-isocrystal for almost all $v$. Morphisms are morphisms in $\mathbf{Og}(K)$.
\end{defn}
Morphisms in $\mathbf{FOg}(K)$ actually respect the filtration by the following:
\begin{lemma} \label{properfil}
With the notation above we have the following properties.
\begin{enumerate}
\item Given an object $\cV$ of $\mathbf{Og}(K)$  there is at most one filtration  $W_{\d}\cV$ on $\cV$ such that $(\cV, W_{\d}\cV)$ is an object of $\mathbf{FOg}(K)$.
\item Let $(\cV, W_{\d}\cV)$, $(\cV^\prime, W_{\d}\cV^\prime)$ be objects of $\mathbf{FOg}(K)$. Then any morphism $\cV\to \cV^\prime$ in $\mathbf{Og}(K)$ respects the filtration and is strict.
 \item $\mathbf{FOg}(K)$ is a $\Q$-linear abelian category.
 \item  Given an object $(\cV, W_{\d}\cV)$ of $\mathbf{FOg}(K)$ for almost all $v$ we have that $\cV_{v}$ has a unique decomposition as a direct sum of pure $F$-$K_{v}$-isocrystal of different weights.
\end{enumerate}
\end{lemma}
\begin{proof} (1) and (4) follow from Lemma \ref{wn}.
Since the morphisms between pure $F$-$K_{v}$-isocrystals of different weights are trivial, it follows that morphisms between mixed  $F$-$K_{v}$-isocrystal with integral  weights respect the filtrations and are strict. Using the fact that $g_{v}$ induces an isomorphism $W_{\d}V_{v} \simeq W_{\d}V\otimes_{K} K_{v}$ and that the map $K\to K_{v}$ is faithfully flat, we get (2).
Assertion (3) follows easily from (2).
\end{proof}

Note that for $(\cV, W_{\d}\cV)\in\mathbf{FOg}(K)$ and an integer $n\in \Z$ we have that
\[(\cV, W_{\d}\cV)(n)\df  (\cV(n), W_{ \d}{}_{+ 2n}\cV(n)) \in\mathbf{FOg}(K)\] 
such that  $\gr_{i}^W\cV(n)\df W_{i+2n}\cV(n)/W_{i+2n-1}\cV(n)$ is pure of weight $i$ (\cf Definition \ref{OgDef}).

\subsection{The weight filtration on $\mathbf{FOg}(K)$}
Consider the Serre subcategories $\mathbf{FOg}(K)_{\leq n} $  of $\mathbf{FOg}(K)$ given by objects $(\cV, W_{\d}\cV)$ of $\mathbf{FOg}(K)$ with $W_{n}\cV=\cV$. We get a filtration
\[\cdots \to \mathbf{FOg}(K)_{\leq n}\longby{\iota_{n}} \mathbf{FOg}(K)_{\leq n+1}\longby{\iota_{n+1}}\cdots \]
Recall that a filtration of an abelian category by Serre subcategories is a weight filtration (in the sense of \cite[Def. D.1.14]{BVK}) if it is separated, exhaustive and split, \ie all the inclusion functors $\iota_{n}$ have exact right adjoints.
\begin{lemma} $\mathbf{FOg}(K)_{\leq n}\subset \mathbf{FOg}(K)$ is a weight filtration.\end{lemma}
\begin{proof}
 In fact, the filtration is clearly separated, \ie $\cap \mathbf{FOg}(K)_{\leq n}=0$, and  exhaustive, \ie  $\cup \mathbf{FOg}(K)_{\leq n}=\mathbf{FOg}(K)$. The claimed adjoints are given by $(\cV, W_{\d}\cV)\mapsto (W_{n}\cV, W_{\d}^{\leq n}\cV)$  where $W_{i}^{\leq n}\cV= W_{i}\cV$ for $i< n$ and $W_{i}^{\leq n}\cV=  W_{n}\cV$ for $i\geq n$ and they are exact by Lemma \ref{properfil} (2).
\end{proof}
We have that \[ \gr_{i}^W\cV=W_{i}\cV/W_{i-1}\cV  \in \mathbf{FOg}(K)_{i}\df \mathbf{FOg}(K)_{\leq i}/\mathbf{FOg}(K)_{\leq i-1} . \]
Note that these categories $\mathbf{FOg}(K)_{i}$ are not necessarily semi-simple.
\bigskip

We may introduce a notion of effectivity following \cite[7.4.2]{An} or \cite[17.4.4]{BVK}.

\begin{defn}\label{fogeff} 
An object $\left(V,  (V_{v}, F_{v})_{v\in \cP'}, (g_{v})_{v\in \cP'}\right)$ of $ \mathbf{Og}(K)$ is said to be  {\it l-effective}\, if there exists a $\cO_K$-lattice $L$ of $V$ such that the image under $g_v$ of $L \otimes_{\cO_K}\cO_{K_v}$ in $V_v$ is preserved by the $F_v$ for almost all $v\in \cP'$. 
Denote $\mathbf{Og}(K)^{\eff}\subset \mathbf{Og}(K)$ the full subcategory of l-effective objects.

Similarly define $\mathbf{FOg}(K)^{\eff}\subset \mathbf{FOg}(K)$ as the full subcategory given by objects  $(\cV, W_{\d}\cV)$ such that $\cV$ is in $\mathbf{Og}(K)^{\eff}$.

Moreover, we say that an object  $(\cV, W_{\d}\cV)$ of $\mathbf{FOg}(K)$ is {\it e-effective} if  the eigenvalues of $F_{v}^{n_{v}}$ on $ V_{v}$ are algebraic integers for almost all $v$.
\end{defn}

 \begin{remarks}\label{r.leff} 
(a) An object $\left(V,  (V_{v}, F_{v})_{v\in \cP'}, (g_{v})_{v\in \cP'}\right)$ of $ \mathbf{Og}(K)$ is l-effective if, and only if,   {\it every} $\cO_K $-lattice  $L$ of $V$ satisfies the condition in Definition \ref{fogeff}. Indeed any two  $\cO_K $-lattices $L,L'$ of $V$ coincide over  $\cO_K[1/n]$ for $n$ sufficiently divisible. Hence if  $L$ satisfies the condition in Definition \ref{fogeff} for any $v\in \cP^{\prime\prime}\subset \cP'$, the same does $L'$ for any $v\in  \cP^{\prime\prime}\cap \cP_{n}$.

(b)  
Since any $\cO_K[1/n]$-lattice $L$ of  $V$ is isomorphic to the base change along $\cO_K\to \cO_K[1/n]$ of a $\cO_K $-lattice,   an object $\left(V,  (V_{v}, F_{v})_{v\in \cP'}, (g_{v})_{v\in \cP'}\right)$ of $ \mathbf{Og}(K)$ is  l-effective if, and only if, there exists a positive integer $n$ and a $\cO_K[1/n]$-lattice $L$ of $V$ such that the image under $g_v$ of $L \otimes_{\cO_K}\cO_{K_v}$ in $V_v$ is preserved by the $F_v$ for  all $v\in \cP_{n}$.

(c)  
If $\cV$  in $\mathbf{Og}(K)$ is l-effective, the same is any subobject $\cV'\subset \cV$  in $\mathbf{Og}(K)$.  As a consequence   given $(\cV, W_{\d}\cV)\in \mathbf{FOg}(K)^{\eff}$ all $W_{i}\cV$ and $\gr_{i}^W\cV$ are in  $\mathbf{FOg}(K)^{\eff}$ as well.
\end{remarks}

According with \cite[Def. 14.3.2]{BVK} we also set a subcategory of level $\leq 1$ objects:
\begin{defn}\label{fog1} Call Artin-Lefschetz objects those  objects of $\mathbf{FOg}(K)_{-2}$ which are $(\cV, W_{\d}\cV)(1)$ for $(\cV,  W_{\d}\cV)$ both l-effective and e-effective of weight zero. Denote $\mathbf{FOg}(K)_\L$ the Serre subcategory of $\mathbf{FOg}(K)_{-2}$ given by Artin-Lefschetz objects.
 Denote $\mathbf{FOg}(K)_{(1)}$ the full subcategory of $\mathbf{FOg}(K)$ given by those  $(\cV,  W_{\d}\cV)$ that are e-effective of weights $\{-2,-1,0\}$ such that $ W_{-2}\cV$ is an Artin-Lefschetz object and $(\cV,  W_{\d}\cV)(-1)$ is l-effective.
\end{defn}

\subsection{The Bost-Ogus category $\mathbf{BOg}(K)$}\label{s.bog}
Let $\mathbf{BOg}(K)$ denote the $\Q$-linear category whose objects are finite dimensional $K$-vector spaces $V$ such that the reduction modulo $\p_{v}$ of $V$ is equipped with a $\sigma_{v}$-semilinear endomorphism for almost every unramified place $v$ of $K$. Morphisms are morphisms of $K$-vector spaces which respect the extra structure. The category $\mathbf{BOg}(K)$ is the category denoted $\mathbf{Frob_{ae}}(K)$ in \cite[2.3.2]{Bo}. 

Also this category can be better described as a $2$-colimit  category.  Recall that, given a positive integer $n$,  $\cP_{n}$ denotes  the subset of $\cP$ consisting of those $v$ such that $n$ is invertible in $\cO_{K_v}$ (see Remark \ref{rem.pn}).  
Let $\mathcal L_{\cP_n}$ be the category whose objects are of the type $\left(V, L, ({}^{\flat}F_{v})_{v\in \cP_n}\right)$ where $V$ is a finite dimensional $K$-vector space, $L$ is an $\cO_K[1/n]$-lattice in $V$ and   ${}^{\flat}F_{v}$ is a $\sigma_{v}$-semilinear endomorphism on  $L\otimes_{\cO_{K}} k_{v}$.  A morphism in $\mathcal L_{\cP_n}$ is the data of a homomorphism of  lattices  (and by $K$-linearization of vector spaces)  which respects the given ${}^{\flat}F_{v}$ for all $v$. We then define
\[\mathbf{BOg}(K):=\colim_{\cP_n \subseteq \cP}\ \mathcal{L}_{\cP_n}.
\]
Note that there is a functor (see also \cite[7.4.2]{An})
\begin{equation}\label{psifun}
\Psi\colon\mathbf{FOg}(K)^{\eff}\to \mathbf{BOg}(K) \end{equation}
defined as follows. 
Consider an object $(\cV, W_\d\cV)$ in $\mathbf{FOg}(K)^{\eff}$ and assume $\cV$  in $\mathbf{Og}(K)^{\eff}$ is represented by  an object  $\left(V,  (V_{v}, F_{v})_{v\in \cP^{\prime}}, (g_{v})_{v\in \cP^{\prime}}\right)$ of   $ \mathcal C_{\cP^{\prime}}$. Let $L$ be any  $\cO_K$-lattice of $V$. By Remarks \ref{r.leff} (a) and \ref{rem.pn}   there  exists a positive integer $n$ such that  $\cP_{n}\subset \cP^{\prime}$ and the image under $g_v$ of $L \otimes_{\cO_K}\cO_{K_v}$ in $V_v$ is preserved by  $F_v$ for all  $v\in\cP_{n}$. 
Let $\Psi (\cV, W_\d\cV)$ be represented by $ \left(V, L\otimes_{\cO_K}\cO_{K}[1/n], ({}^{\flat}F_{v})_{v\in \cP_n}\right)$ in $\mathcal L_{\cP_{n}}$ where ${}^{\flat}F_{v}$ is the reduction modulo $\p_{v}$ of the mapping on $L \otimes_{\cO_K}\cO_{K_v}$ induced by $F_{v}$.  Note that a different choice of a lattice   of $V$ will provide the same lattice over $\cO_{K}[1/n]$ for $n$ sufficiently divisible and hence the functor is well defined.
Furthermore, remark that the functor $\Psi$ is not full (\cf Remark \ref{psibog}).

\section{Logarithms and universal extensions}

\subsection{The $p$th power operation}
Recall from \cite[II \S7 n.2 p. 273]{DG} and \cite[Exp. ${\rm VII}_{\rm A}$, \S 6]{sga3} that given a field $k$ of characteristic $p>0$ and a $k$-group scheme $G$ one can define a  $p$th power operation $x \mapsto x^{[p]}$ on $\Lie(G)$ as follows. 
Recall that  \[\Lie (G)=\mathrm{Ker}(G(k[\varepsilon]/(\varepsilon^2))\to G(k))\]
 and  for any $x\in \Lie(G)$ write $e^{\varepsilon x}$ for the corresponding element in $G(k[\varepsilon]/(\varepsilon^2))$.
Let $k[\varsigma,\pi]\subseteq k[\varepsilon_{1},\dots,\varepsilon_{p}]/(\varepsilon^2_{1},\dots, \varepsilon^2_{p})$, with $\varsigma=\sum_{i=1}^p\varepsilon_{i}, ~\pi=\prod_{i=1}^p \varepsilon_{i}$, be the subalgebra generated by the elementary symmetric polynomials in $\varepsilon_{i}$. Observe that $\varsigma^p=0$, $\varsigma\pi=0$ and $\pi^2=0$. Then $e^{\varepsilon_{1} x}e^{\varepsilon_{2} x}\cdots e^{\varepsilon_{p} x}$ makes sense as element in \[\mathrm{Ker}(G(k[\varepsilon_{1},\dots,\varepsilon_{p}]/(\varepsilon^2_{1},\dots, \varepsilon^2_{p}))\to G(k))\]
 where we use the multiplicative notation for the group law on $G$. Since  $e^{\varepsilon_{1} x}e^{\varepsilon_{2} x}\cdots e^{\varepsilon_{p} x}$ is invariant by permutations of the $\varepsilon_{i}$'s (see \cite[II \S4, 4.2 (6) p. 210]{DG}), it is indeed an element of $\mathrm{Ker}(G(k[\varsigma, \pi])\to G(k))$. Consider now the canonical projection $k[\varsigma, \pi]\to k[\pi], \varsigma\mapsto 0$.
It induces a map \[G(k[\varsigma, \pi])\to G(k[\pi]).\] 
Let $e^{\pi y}$ be the image of $e^{\varepsilon_{1} x}e^{\varepsilon_{2} x}\cdots e^{\varepsilon_{p} x}$  via this map. We further have that $e^{\pi y}$ is mapped to the  unit section via the map $G(k[\pi])\to G(k)$ induced by $\pi\mapsto 0$. We have
\[\xymatrix{G(k[\varsigma,\pi])\ar[r]\ar[d]& G(k)\\
G(k[  \pi])\ar[ur]&
}\qquad
\xymatrix{e^{\varepsilon_{1} x}e^{\varepsilon_{2} x}\cdots e^{\varepsilon_{p} x}\ar[r]\ar[d]&1\\
e^{\pi y}\ar[ur]&}
\]
Hence $y\in \Lie(G)= \mathrm{Ker}(G(k[\pi])\to G(k))$ and we define $x^{[p]}:=y $.
The map \begin{equation}\label{pth}
{}^{[p]}\colon  \Lie(G) \to  \Lie(G)\ \ \ x\mapsto x^{[p]}
\end{equation}
endows $\Lie(G)$ with a structure of Lie $p$-algebra over $k$ (see \cite[II, \S 7 Prop. 3.4 p. 277]{DG}). 
If $G$ is  commutative, then  $[x,y]=0$  in $\Lie(G )$ (see \cite[Exp. II, Def. 4.7.2]{sga3}) and hence ${}^{[p]}$ is $p$-linear, \ie $(x+y)^{[p]}=x^{[p]}+y^{[p]}, (\lambda x)^{[p]}=\lambda^px^{[p]} $ for $\lambda\in k$ and $x,y\in \Lie(G )$.
Up to the usual identification of $\Lie(G)$ with the invariant derivations of $G$, the  $p$th power operation maps a derivation $D$ to $D^p $ (\cf \cite[Exp. VII${}_{\rm A}$, \S 6.1]{sga3}, \cite[II, \S 7, Prop. 3.4 p. 277]{DG}).

Let $\sigma$ denote the Frobenius map on $k$. For the sake of exposition we provide a proof of the following well known fact (\cf  \cite[Exp. VII${}_A$ \S 4]{sga3}). 
\begin{lemma}\label{l.pth}
Let $G$ be a commutative algebraic $k$-group. Then the  $p$th power operation \eqref{pth} is a  $\sigma$-semilinear map and coincides with the map on Lie algebras associated to the Verschiebung $\mathrm{Ver}_{G}\colon G^{(p)} \to G$.
\end{lemma}
\begin{proof} The first claim is obvious.  By  functoriality of the Veschiebung  (see \cite[Exp. ${\rm VII}_{\rm A}$,  4.3]{sga3}) and the fact that $\Lie(\mathrm{Fr}_{G})=0$ with $ \mathrm{Fr}_{G}\colon G\to G^{(p)}$ the (relative) Frobenius,   it suffices to show the second claim replacing $G$ by the kernel of $\mathrm{Fr}_{G}$: we thus assume $G$ to be finite (infinitesimal). By \cite[Thm. 3.1.1]{BBM} we can embed $G$ as a closed subgroup-scheme of an abelian variety. By functoriality of the Verschiebung and of the $p$th power operation we can further reduce to the case of abelian varieties. The latter follows from  Example \ref{ex.gagm} (b) below as duality on abelian varieties exchanges Frobenius with Verschiebung (\cf \cite[Prop. 7.34]{GM}).
\end{proof}

\begin{examples}\label{ex.gagm}
(a) It follows from \cite[II, \S 7 Exemples 2.2 p. 273]{DG} that the $p$th power operation on  $\Lie(\G_{a})$ is the zero map $x\mapsto 0$, while  the $p$th power operation on $\Lie(\G_{m})$ is given by $x\mapsto x^p$. 

 (b) When  $G=A$ is an abelian variety and $A^*$ is the dual abelian variety,  it is known that there is a natural isomorphism
$\Lie(A)\simeq \mathrm H^1(A^*, \cO_{A^*})$ and the  $p$th power operation on $\Lie(A)$ corresponds to the Frobenius map on $\mathrm H^1(A^*, \cO_{A^*})$, \ie  the $\sigma$-semilinear map induced by the Frobenius homomorphism $\alpha\mapsto \alpha^p$ on $\cO_{A^*}$ (see \cite[\S 15, Thm. 3]{Mu}). 
\end{examples}
\smallskip

\subsection{The logarithms}
Let $k$ be a finite field of characteristic $p$.   Let $W(k)$ be the  ring of Witt vectors over $k$,  $K_{0}$  its quotient field,  and set $W_{n}(k)\df W(k)/(p^{n})$. Let $G$ be a group scheme of finite type over $W(k)$.
As claimed in \cite[III, 5.4.1]{ega} we have that
\[ G(W(k)) = \varprojlim G(W_{n}(k)). \] 
(The separatedness hypothesis in \loccit
can be ignored since $W(k)$ is local.)  
 
Consider the canonical homomomorphism of groups
\begin{equation}\label{rhov}
\rho \colon G(W(k))\to G(k)
\end{equation}
induced by the closed immersion $\spec(k)\to \spec (W(k))$.
Note that as $k $ is a finite field, the $k $-valued points $G(k)$ of $G$ form a finite, and hence torsion, group.
In particular, if we denote $\Gamma $ the kernel of $\rho $ in \eqref{rhov}, we obtain an isomorphism of $\Q$-vector spaces
\begin{equation} \label{gg} G(W(k))\otimes_{\Z} \Q \simeq  \Gamma \otimes_{\Z} \Q.
\end{equation}
\smallskip

Assume from now on that $G$ is a {\it smooth} and {\it commutative} $W(k)$-group scheme and let $G_{n}$ be the base change of $G$ to $S_n=\spec (W_n(k))$; in particular  $G_1$ denotes the special fiber of $G$. Let $\mathcal J$ be the ideal sheaf of the unit section of $G_{n}$ and recall that $\cO_{G_n}/\mathcal J^N $ is a finite and free $W_n(k)$-module for any $N>0$.  
Hence  $D=\varinjlim_N\Hom_{W_n(k)\text{-mod}}(\cO_{G_n}/\mathcal J^N, W_n(k)) $ is a coalgebra with a $W_{n}(k)$-algebra structure induced by the group structure on $G_{n}$. The flatness of $D$ over $W_{n}(k)$ ensures that the PD-structure on $pW_{n}(k)$ extends uniquely to a PD-structure $(\gamma_m)$ on $pD $ and hence there exist 
two mutually inverse maps 
\begin{equation*}
\exp \colon pD \to (1+pD )^*,\qquad \log\colon(1+pD )^* \to pD , 
\end{equation*}
defined by   $\exp(x)=\sum_{m\geq 0}\gamma_m(x) $ and $\log(1+x)= \sum_{m\geq 1}(-1)^{m-1}(m-1)!\gamma_m(x)$ (\cite[III, 1.6]{Me}).  Let $\mathrm{Cospec}(D)(W(k))\subset D $ denote the subgroup of $W_{n}(k)$-algebra homomorphisms. Then $p \mathrm{Cospec}(D)(W(k))=  \mathrm{Cospec}(D)(W(k))\cap pD$ consists of those homomorphisms whose reduction modulo $p$ is the homomorphism associated to the unit section of the special fiber of $G$.  Let $\mathrm{Prim}(D)\subset D$ consists of the primitive elements of the coalgebra $D$, \ie those $x\in D$ such that $\Delta(x)=x\otimes 1+1\otimes x$ with $\Delta$ the comultiplication of $D$, and let $p\mathrm{Prim}(D)=\mathrm{Prim}(D)\cap pD$. Then by \cite[III, 2.2.5]{Me} (see also \cite[\S 5.2]{ABe})  there is an  isomorphism of groups $\exp_{G,n}$ which makes the following diagram
\begin{equation}\label{dprim}
\xymatrix{
\ker\big( \Lie(G_n)\to \Lie(G_1)\big) \ar[r]^{\simeq}\ar@{.>}[d]^{\exp_{G,n} }&\ar[d] p\mathrm{Prim}(D)\subseteq pD\ar[d]^{\exp}  \\
\ker\big( G_n(W_n(k))\to G_n(k)\big) \ar[r]^{\simeq}&p(\mathrm{Cospec}(D)(W(k))\subset (1+pD)^*
}
\end{equation}
commute. The vertical arrow on the left can also  be written as an isomorphism 
\[\exp_{G,n}\colon    p\Lie(G_{n}) \longby{\simeq}  \ker\big( G(W_n(k))\to G(k)\big) .\] 
Finally, taking the limit over $n$, one gets the exponential isomorphism for $G$
\begin{equation*}
\exp_{G }\colon  p\Lie(G) \longby{\simeq}  \Gamma . \end{equation*}  
Let  $\log_{G}\colon \Gamma\longby{\simeq}  p\Lie(G)$  denote the inverse of $\exp_G$.  
We set (\cf \cite[\S 2.4, p. 169]{Ta}):

\begin{defn} \label{log2} The logarithm is the isomorphism of $\Q$-vector spaces
\[\log_{G,\Q}\colon G(W(k)) \otimes_{\Z } \Q  \longby{\simeq} \Lie(G)\otimes_{W(k)} K_0\]
obtained by composing \eqref{gg} with  $\log_{G}\otimes {\rm id}_{\Q}$  and recalling that $  p \Lie(G)\otimes_\Z\Q\simeq \Lie(G)\otimes_\Z\Q \simeq \Lie(G)\otimes_{W(k)} K_{0}$
where the isomorphism $  p \Lie(G)\otimes_\Z\Q\longby{\simeq} \Lie(G)\otimes_\Z\Q$ is the $\Q$-linearization of the inclusion $p\Lie(G)\to \Lie(G)$.
\end{defn}

\begin{examples}\label{ex.ga} 
(a) Let $G=\G_{a, W(k)}$. Then  \eqref{rhov}  is the reduction map $ W(k)\to k$ and hence  $\Gamma=pW(k)$. Recall that $\Lie(\G_{a,W_{n}(k)})=\ker( W_{n}(k)+ W_{n}(k)\varepsilon\to W_{n}(k), a+b\varepsilon  \mapsto a)\simeq  W_{n}(k)\varepsilon $ with $\varepsilon^2=0$  and $D= \Hom_{{\rm cont}}(W_{n}(k)[[Z]], W_{n}(k))$.  
Diagram \eqref{dprim} becomes
\[
\xymatrix@R=5mm{
pW_n(k)\varepsilon \ar[r]^{\simeq}\ar@{.>}[d]^{\exp_{\G_a,n}} & p\mathrm{Prim}(D) \ar[d]^{\exp} &&   b\varepsilon  \ar@{|->}[r]&f_b\ar@{|->}[d]\\
pW_{n}(k) \ar[r]^(0.35){\simeq}&p(\mathrm{Cospec}(D)(W(k)) && \exp(f_b)(Z)&\ar@{|->}[l] \exp(f_b) 
} 
\]
where $f_b(1)=0, f_b(Z^r)=rb$ for $r\geq 1$. Since $\gamma_m(f_b)(Z)=0$ for $m\neq 1$ and $\gamma_1(f_b)=f_b$, one gets that $\exp(f_b)(Z)=b$ and hence,  up to the obvious identifications, we may consider $\exp_{\G_a}$ and $\log_{\G_a}$ as the identity maps on $W_n(k)$. Finally $\Lie_{\G_{a},\Q}$  is the identity of $K_0$, up to the usual identifications of $G(W(k))$ and $\Lie(\G_{a})$ with $W(k)$, \ie  $\log_{\G_a,\Q}\colon W(k)\otimes_\Z\Q\to \varepsilon W(k)\otimes_{W(k)}K_0$ in Definition \ref{log2} is $x \mapsto \varepsilon x$.

(b) Let $G=\G_{m, W(k)}=\spec (W(k)[X^{\pm 1}])$. Then   $\Gamma=1+pW(k) \subset W(k)^*$, $\Lie(\G_{m,W_{n}(k)}) \simeq 1+pW_{n}(k)\varepsilon $ with $\varepsilon^2=0$, and $D$ is as in (a) as a co-algebra (with $Z=X-1$).  Diagram \eqref{dprim} becomes
\[
\xymatrix@R=5mm{
1+pW_n(k)\varepsilon \ar[r]^{\simeq}\ar@{.>}[d]^{\exp_{\G_m,n}} & p\mathrm{Prim}(D) \ar[d]^{\exp} &&   1+b\varepsilon  \ar@{|->}[r]&f_b\ar@{|->}[d]\\
1+pW_{n}(k) \ar[r]^(0.4){\simeq}&p(\mathrm{Cospec}(D)(W(k)) &&1+ \exp(f_b)(Z)&\ar@{|->}[l] \exp(f_b) 
}
\]
where $f_b(1)=0, f_b(Z^r)=rb$ for $r\geq 1$. Since $\gamma_m(f_b)$ maps $Z$ to $0$ if $m=0$ and to $b^m/m!$ if $m\geq 1$ then $\exp(f_b)(Z)=\sum_{m\geq 1}b^m/m!$ and  $\exp_{\G_m,n}(1+b\varepsilon)= \sum_{m\geq 0}b^m/m!$.  Hence, up to the obvious identifications,   $\exp_{\G_m}$ is the exponential map $pW(k)\to 1+pW(k)$ and  $\log_{\G_m} $ is the  usual $p$-adic logarithm. Finally the isomorphism $\log_{\G_m,\Q}\colon W(k)^*\otimes_\Z\Q\to 1+( W(k)\varepsilon \otimes_{W(k)}K_0)$ in Definition \ref{log2} is given by 
\[x\otimes 1\mapsto 1+ \frac{\log(1+y)\varepsilon}{(p_v^{n_v}-1)}\] where $x^{p_v^{n_v}-1}=1+y$, $y\in pW(k)$.
\end{examples} 
 
\begin{remark}\label{gal}
Note that for any $W_n(k)$-scheme $T$, one can define a functorial (in $T$ and $G$) isomorphism $\exp_{G,n}\colon p\uLie(G)(T)\longby{\simeq} \ker(G(T)\to G(T_0))$ where $T_0$ is the reduction of $T$ modulo $p$ and $\uLie(G)$ is the Lie  algebra  scheme of $G$ (see \cite[\S 5.2]{ABe}). In particular any map $\exp_{G,n}$ (and thus $\exp_G$) behaves well with respect to finite unramified extension of $W(k)$. Further, for any finite Galois extension $k'/k$ the map $\log_{G,\Q}$ in Definition \ref{log2} can be obtained by descent from the analogous isomorphism over $W(k')$. 
\end{remark}
Let  $u\colon L\to G$ be a morphism of $W(k)$-group schemes where $L = \Z^r$ and $G$ is a smooth and commutative $W(k)$-group scheme with connected fibers. Define $\delta_{u}$ as the $K_0$-linear extension of the composition 
\begin{equation}\label{eq.del}
L(W(k))\longby{u\otimes 1}  G(  W(k))\otimes_\Z \Q \longby{\log_{G,\Q}} \Lie(G)\otimes_{W(k)}K_0 
\end{equation} 
\begin{lemma}\label{d.lbd} Let  $u\colon L\to G$ be a morphism of $W(k)$-group schemes where $L$ is a lattice (\ie isomorphic to $\Z^r$ over some finite unramified extension of $W(k)$) and $G$ is a smooth, connected and commutative $W(k)$-group scheme.   Then there is a unique morphism of $K_0$-vector spaces 
\[\delta_{u}\colon \Lie(L\otimes\G_{a})\otimes_{W(k)}K_0\to \Lie(G)\otimes_{W(k)}K_0 \] 
which is functorial in $u$ and in $W(k)$ and agrees with the one in \eqref{eq.del} for $L$ constant.  Moreover, if $u$ is given by the obvious inclusion $\mathrm{id}\otimes 1 \colon L\to L\otimes \G_{a}$, then $\delta_{u}$ is the identity of $\Lie(L\otimes\G_{a})\otimes_{W(k)}K_0 $.
 \end{lemma}

\begin{proof}  It follows by Remark \ref{gal}. The last assertion follows from the fact that, up to the usual identifications,  $\log_{\G_a,\Q}$ in Definition \ref{log2} is the identity  map as computed in Example \ref{ex.ga} (a).
\end{proof}

\begin{remark}\label{r.for}
For any  formal $W(k)$-group scheme $\widehat G=\varprojlim_{n} G_{n}$ where $G_{n}$ is a smooth commutative group scheme of finite type over $W_{n}(k)$,
one can define in a similar way the logarithm  $\log_{\widehat G, \Q}\colon \widehat G(W(k)) \otimes_{\Z } \Q \longby{\simeq} \Lie(\widehat G)\otimes_{\Z_p} \Q_p $ with $\Lie(\widehat G):=\varprojlim_n \Lie(G_{n})$. 
This construction is again functorial, \ie  given a morphism $g=\varprojlim_{n} g_{n}\colon \widehat G\to \widehat H$ then  $\log_{\widehat H,\Q}\circ (g\otimes 1)=(\Lie(g)\otimes \mathrm{id})\circ \log_{\widehat G,\Q}\colon \widehat G(W(k)) \otimes_{\Z} \Q\to \Lie(\widehat H)\otimes_{\Z_p} \Q_p$.
Further, $\log_{G,\Q}$ in Definition \ref{log2} equals $\log_{\widehat G,\Q}$  with $\widehat G$ the $p$-adic completion of $G$. 

 As in Lemma \ref{d.lbd}, given a lattice $L$ over $W(k)$ with base change $L_n$ to $W_n(k)$ and a compatible system of morphisms of $W_n(k)$-group schemes $\{u_n\colon L_n\to G_n\}_{n\in \N}$ over $W_n(k)$ we get a natural morphism of $K_0$-vector spaces $\delta_{u}\colon \Lie(L\otimes\G_{a})\otimes_{W(k)}K_0\to \Lie(\widehat G)\otimes_{W(k)}K_0 $.
\end{remark} 

\subsection{Universal extensions}

Let $\mathcal M_{1 }(S)$ be the category of (Deligne) $1$-motives over a scheme $S$ (\cf  \cite[\S 1.2 \& App. C]{BVK}). If $S=\spec (K)$, let $\MQ$ denote the $\Q$-linear category of  $1$-motives up to {\it isogenies} over $K$, \ie  the category whose objects are \[{\sf M}_{K}=[{\sf u\colon L}_{K}\to {\sf G}_{K}]\] in $\M\df \mathcal M_{1 }(\spec (K))$ and whose morphisms are given by $\Hom_{\M}({\sf M}_{K},{\sf N}_{K})\otimes_\Z \Q$. See \cite[Prop. 1.2.6]{BVK} for a proof that $\MQ$  is an abelian category. A morphism $(h_{-1},h_{0})$ in  $\M$ becomes an isomorphism in $\MQ$ if and only if it is an isogeny, \ie $h_{-1}$ is injective with finite cokernel and $h_{0}$ is an isogeny (see \cite[Lemma 1.2.7]{BVK}). Recall that the canonical weight filtration of $1$-motives yields a weight filtration on $\MQ$ (see \cite[Prop. 14.2.1]{BVK}).
We let
 ${\sf A}_{K}$ be the maximal abelian quotient of ${\sf G}_{K}$ and let \[{\sf M}_{\ab,K}\df[{\sf u}_{{\sf A}_{K}}\colon {\sf L}_{K}\to {\sf A}_{K}]\]
 denote the (Deligne) $1$-motive induced by ${\sf M}_{K}$ via  ${\sf L}_{K}\to {\sf G}_{K}\to {\sf A}_{K}$.

Let $\mathcal M_{0,\Q}$ be the abelian category of Artin motives identified with
the full subcategory of $\MQ$ whose objects are of the type $[{\sf L}_{K}\to 0]$, \ie $1$-motives which are pure of weight zero.

Let $\Mal$ be the larger category of  generalized   $1$-motives with {\it  additive factors} over $K$ whose objects are two terms complexes (in degree -1, 0) ${\sf M}_{K}=[{\sf u}_{K}\colon {\sf L}_{K}\to {\sf E}_{K}]$  where ${\sf L}_{K}$ is a a lattice,  ${\sf E}_{K}$ is a
commutative connected algebraic $K$-group and ${\sf u}_{K}$ is a morphism of algebraic $K$-groups. Note that $\Mal$ is a full subcategory of the category $\Ma$ considered in \cite[\S 1]{BVB}. 

\medskip

Let ${\sf M =[u\colon  L\to  G}]\in \mathcal M_{1 }(S)$ be a $1$-motive over $S$. 
Recall (\eg see \cite[\S 2]{ABV}) that there exists the {\it universal $\G_a$-}extension of ${\sf M}$ which we denote by
 \[{\sf M}^{\natural}\df [{\sf u^{\natural} \colon  L} \to   {\sf G^{\natural}}].\]
 It is an extension of ${\sf M}$ by  a vector group ${\mathbb V}({\sf M})$  such that the homomorphism of push-out
 \begin{equation*}
 \mathrm{Hom}_{ \cO_S}({\mathbb V}({\sf M}),W   )\longrightarrow \Ext({\sf M} ,W)
\end{equation*}
 is an isomorphism for all vector groups $W$ over $S$ (where $\mathrm{Hom}_{\cO_S}(-,- )$ means homomorphisms of vector groups). If $S=\spec\, K$ then ${\sf M}_{K}^{\natural}\in \Mal$ is a $1$-motive with additive factors over $K$    (and $\Ext$ is taken in $\Mal$).
\begin{defn}\label{Tdr} The de Rham realization of ${\sf M}$ is
\[T_{\dr}({\sf M}):=\Lie({\sf G}^{\natural}).  \]
\end{defn}
It is easy   to check that ${\mathbb V}({\sf M})$ has to be the vector group associated to $\mathrm{Ext}({\sf M},\G_a)^\vee$. Furthermore $\V({\sf M})$ is canonically isomorphic to the vector group associated to the sheaf $\omega_{{\sf G}^{\ast} }$ of invariant differentials of the semiabelian scheme ${\sf G}^{\ast} $ Cartier dual of ${\sf M}_\ab$. We have a push-out diagram
 \begin{eqnarray}\label{dia.ue}
\xymatrix{
 0\ar[r]&\V({\sf G})=  \omega_{{\sf A}^{\ast}}\ar[r]\ar@{^{(}->}[d]^{i} &{\sf A}^{\natural}\times_{\sf A} {\sf G} \ar[r]  \ar@{^{(}->}[d] &
{\sf G}\ar[r]\ar@2{-}[d] &0\\
0\ar[r]&\V({\sf M})=\omega_{{\sf G}^{\ast}}  \ar[r]\ar@{->>}[d]^{\bar\tau}& {\sf G}^{\natural}
\ar[r]^\rho\ar@{->>}[d]^\tau &
{\sf G}\ar[r] &0\\
& {\sf L}\otimes \G_a \ar@2{-}[r] & {\sf L}\otimes \G_a  & &
}\end{eqnarray}
with ${\sf A}^{\natural} \df \Pic^{\natural , 0}({\sf A}^{\ast})$ the universal extension of ${\sf A}$. Note that the lifting ${\sf u}^{\natural}$ of ${\sf u}$ composed with $\tau$ gives the map ${\sf L}\to {\sf L}\otimes \G_a, x\mapsto x\otimes 1$; see \cite[\S 2]{Ber} for details.

Any morphism of $1$-motives $\varphi\colon {\sf M} \to {\sf N}$ provides a morphism $\varphi^{\natural}\colon {\sf M}^{\natural}\to {\sf N}^{\natural}$ that maps term by term the elements of the corresponding diagrams \eqref{dia.ue}; it provides a morphism such that the induced morphism $\V({\sf M})\to \V({\sf N})$ corresponds to the pull-back on invariant differentials along the induced morphism  obtained via Cartier duality from ${\sf M}_\ab\to {\sf N}_\ab$.

\begin{remark}\label{r.tau}
Assume $S=\spec( W(k))$, and recall the morphism ${\sf u}^\natural\colon {\sf L}
\to {\sf G}^\natural$ defining ${\sf M}^\natural$. By Lemma \ref{d.lbd} we have a morphism of $K_0$-vector spaces $$\delta_{{\sf u}^\natural}\colon \Lie({\sf L}\otimes\G_{a})\otimes_{W(k)}K_0 \to  \Lie({\sf G}^\natural)\otimes_{W(k)}K_0$$ which is a section of $\Lie(\tau)\otimes \mathrm{id}_{K_0}$ with $\tau$ as in \eqref{dia.ue}. We will show next that these $K_0$-vector spaces are endowed with a structure of an $F$-$K_0$-isocrystal and that $\delta_{{\sf u}^\natural}$ commutes with the Frobenius, see Lemma \ref{section}. \end{remark}

\section{Fullness of the Ogus realization}
\subsection{Models}
By writing  $K=\varinjlim_{n} \cO_{K}[1/n]$ any scheme $X_{K}$ of finite type over $K$ admits a model $X[1/n]$  of finite presentation over $\cO_{K}[1/n]$ for $n$ sufficiently divisible, \ie large enough in the preorder given by  divisibility (see \cite[IV, 8.8.2(ii) p. 28]{ega}). Furthermore this model is essentially unique, \ie  models $X[1/n]$ and $X[1/m]$ of $X_{K}$ become isomorphic on $ \cO_{K}[1/N]$, with $N$ a suitable multiple of $m$ and $n$ (see \cite[IV, 8.8.2.5 p. 32]{ega}).

Any algebraic $K$-group $G_{K}$ thus admits a model $ G$ which is a group scheme of finite presentation over  $S=\spec(\cO_{K}[1/n])$, for $n$ sufficiently divisible, and  $G$ is essentially unique.
Note that   $G$  may be assumed to be smooth, see \cite[IV, Proposition 17.7.8(ii)]{ega}.     
Let $G_{\cO_{K_{v}}}$  denote the base change of $G$ to $\cO_{K_{v}}$, when it makes sense, and let $G_{k_{v}}$ be its special fiber.

A morphism of algebraic $K$-groups $f_{K}\colon G_{K}\to G_{K}^{\prime}$  extends to an $S$-morphism of {\it schemes} between the models (see  \cite[IV, 8.8.1.1 p. 28]{ega}). Up to inverting finitely many primes, one can assume that this is indeed a morphism of $S$-group schemes.

Let $\mathbf{GCC}(K)$ be the category of commutative connected algebraic $K$-groups and $\mathbf{GCC}(K)_\Q$ its localization at the class of isogenies.
One can define a functor

\begin{equation}\label{bf}
\bLie\colon \mathbf{GCC}(K)\to \mathbf{BOg}(K)
\end{equation} 
which associates to any commutative connected algebraic $K$-group $G_{K}$ the object in  $\mathbf{BOg}(K)$  (see \S \ref{s.bog}) represented by the triple $(V,L,({}^{\flat}F_{v})_{v\in \cP_{n}})$ in $\mathcal L_{\cP_{n}}$ defined as follows. Set $V\df \Lie(G_{K})$, 
$L\df\Lie(G)$  and ${}^{\flat}F_{v}$ on $\Lie(G_{k_{v}})\cong L\otimes_{\cO_{K}} k_{v}$ the canonical $\sigma_{v}$-semilinear homomorphism described in \eqref{pth} and induced by the Verschiebung $\mathrm{Ver}_{G_{k_{v}}}$ (see Lemma \ref{l.pth}). 

Note that $L$ is a $\cO_{K}[1/n]$-lattice in $\Lie (G_K)$ via the isomorphism  $L\otimes_{\cO_{K}} K\cong \Lie(G_{K})$ and that $\Lie(G_{k_{v}})\cong L\otimes_{\cO_{K}} k_{v}$ is the canonical isomorphism (see \cite[Exp. ${\rm II}$, Prop. 3.4 and \S 3.9.0]{sga3}).

We have:

\begin{thm}[Bost] \label{bost} The functor $\bLie\colon \mathbf{GCC}(K)_{\Q}\to \mathbf{BOg}(K)$  is fully faithful.
\end{thm}
\begin{proof}
The difficult part in the proof is the fullness. See \cite[Thm.~2.3 \& Cor.~2.6]{Bo} and \cite{C-L} for details.
On the other hand, the proof of the faithfulness is immediate. Let $f_{K}\colon G_{K}\to G^{\prime}_{K}$ be a morphism of (commutative connected) algebraic $K$-groups and let $f\colon G\to G^{\prime}$ be a morphism of $\cO_{K}[1/n]$-group schemes which extends it, for $n$  sufficiently divisible. If $\Lie(f_{K})=0$ then $f_{K}=0$, see \cite[II, \S 6, n. 2, Prop. 2.1 (b)]{DG}.
\end{proof}

Now consider a $1$-motive ${\sf M}_{K}$ over $K$ and models ${\sf M}$ over  $S=\spec(\cO_{K}[1/n])$, \ie   $1$-motives ${\sf M}$ over $S$ such that the base change to $K$ is isomorphic to ${\sf M}_{K}$, for $n$ sufficiently divisible. We have:

\begin{lemma}\label{1mod} Any $1$-motive   ${\sf M}_{K}=[{\sf u}_{K}\colon {\sf L}_{K}\to {\sf G}_{K}]$ over $K$  admits a model ${\sf M}=[{\sf u}\colon {\sf L}\to {\sf G}]$ over an  $S=\spec(\cO_{K}[1/n])$  for $n$ sufficiently divisible. This model is essentially unique, \ie any two such models are isomorphic over a $ \spec(\cO_{K}[1/m])$ with $n|m$. 
\end{lemma} 
\begin{proof}
Let ${\sf T}_{K}$  be the maximal torus in ${\sf G}_{K}$ and let ${\sf A}_{K}$ be the maximal abelian quotient of ${\sf G}_{K}$. Let $\sf T$ be the model of ${\sf T}_{K}$ over $S$.  We may assume that it is a torus. Indeed ${\sf T}_{K}$  becomes split over a  finite Galois extension $K'$ of $K$ and, up to enlarging $n$, we may assume that $S'=\spec (\cO_{K'}[1/n])$ is \'etale over $S = \spec(\cO_{K}[1/n])$. Now,  $\sf T$ is a torus since its base change to $S'$ is a split torus  by the essential unicity of the model. Further, by Cartier duality, the group of characters of any torus over $K$ admits a model over an $S$ which is \'etale locally a split lattice of fixed rank. In particular ${\sf L}_{K} $ extends to such a model $\sf L$ over $S$. On the other side, we may assume that the model $\sf A$ of ${\sf A}_{K}$ is   an abelian scheme over $S$ \cite[\S 1.2, Thm. 3]{BLR}, hence $\sf A$ is the global N\'eron model of  ${\sf A}_{K}$ over  $S$. Finally, by  \cite[\S 10.1, Prop. 4 \& 7]{BLR} ${\sf G}_{K}$ admits a global N\'eron lft-model $\mathcal G$ over $S$ and the identity component $\sf G:=\mathcal G^{0}$ is a smooth  model of $\sf G_{K}$  of finite type over $S$. 

Recall that $\sf G_{K}$ is the Cartier dual of the $1$-motive  ${\sf M}_{\ab,K}^\ast=[{\sf u}_{{\sf A}_{K}}^\ast \colon {\sf Y}_{K}\to {\sf A}_{K}^* ]$ where ${\sf Y}_{K}$ is the group of characters of ${\sf T}_{K}$. By the above discussion on models of abelian varieties and lattices and by the universal property of N\'eron models \cite[\S 1.2, Definition 1]{BLR}, the $1$-motive  ${\sf M}_{\ab,K}^\ast$ extends uniquely to a $1$-motive ${\sf M}_{\ab}^\ast=[ {\sf u}_{{\sf A}}^\ast\colon {\sf Y}\to {\sf A}^\ast]$ over $S$  and, by the essential unicity of the model of $\sf G_{K}$, $\sf G$ is the Cartier dual of ${\sf M}_{\ab}^\ast$, in particular it is extension of the abelian scheme $\sf A$ by the torus $\sf T$.

Clearly ${\sf u}_{K}$ extends uniquely to a morphism ${\sf u}\colon {\sf L}\to \mathcal G$. It remains to check that up to enlarging $n$, $ \sf u$  factors through the open subscheme ${\sf G}=\mathcal G^{0}$ of $\mathcal G$. Since the formation of lft N\'eron models is compatible with \'etale base change we may assume that $\sf T$ is a split torus over $S$ and ${\sf L}\simeq \Z^r$.
Fix a basis of $\Z^r$ and let $e_i\colon S\to \Z^r,i=1,\dots, r,$ denote the corresponding morphisms. It is sufficient to check that up to enlarging $n$ each map ${\sf u}\circ e_i$ (which corresponds to a $K$-rational point of ${\sf G}_{K}$) factors through $\sf G$. This follows from  \cite[IV, 8.8.1.1]{ega}. 
\end{proof}

\subsection{The Ogus realization}
For a $1$-motive   ${\sf M}_{K}=[{\sf u}_{K}\colon {\sf L}_{K}\to {\sf G}_{K}]$ over $K$ consider its model ${\sf M}=[{\sf u}\colon {\sf L}\to {\sf G}]$ over a  $S=\spec(\cO_{K}[1/n])$ given in Lemma \ref{1mod}.  For any $v\in \cP_{n}$ let ${\sf M}_{\cO_{K_{v}}}$ be the base change to $\cO_{K_{v}}$ of the model  ${\sf M} $ and let ${\sf M}_{k_{v}}$ be the reduction of ${\sf M}_{\cO_{K_{v}}}$ modulo $\p_{v}$.

\begin{defn}\label{TOg}
Let $V\df T_{\dr}({\sf M}_{K})$ be the de Rham realization (see Definition \ref{Tdr}) as a $K$-vector space. Let $V_{v}\df T_{\dr} ({\sf M}_{\cO_{K_{v}}} )\otimes_{\cO_{K_{v}}} K_{v}$ together with the induced isomorphism obtained by $g_{v}\colon V\otimes_{K} K_{v} \to V_{v}$ and compatibility of $T_{\dr}$ with base change. Consider a model ${\sf M}$ over  $S=\spec(\cO_{K}[1/n])$ as in Lemma \ref{1mod}.  For every $v\in \cP_{n}$ consider the canonical isomorphism $T_{\dr}({\sf M}_{\cO_{K_{v}}})\simeq T_\cris({\sf M}_{k_{v}} )$ as stated  in \cite[Cor. 4.2.1]{ABV}. Via this identification $T_{\dr}({\sf M}_{\cO_{K_{v}}})$ is endowed with a $\sigma_{v}^{-1}$-semilinear endomorphism $\Phi_{v}$ (Verschiebung). Let $F_{v}$ be the $\sigma_{v}$-semilinear endomorphism
$(\Phi_{v}\otimes \mathrm{id}_{K_{v}})^{-1}$ on  $V_{v}$.  In this way, we have associated to ${\sf M}_{K}$ an object $\cV\df \left(V, (V_{v}, F_{v})_{v\in \cP_{n}}, (g_{v})_{v\in \cP_{n}}\right)$
 in $\mathbf{Og}(K)$.
The usual weight filtration on $1$-motives induces
a weight filtration  $W_{\d}\cV$ on $\cV$ in $\mathbf{Og}(K)$. We shall denote \[T_{\og}({\sf M}_{K})\df (\cV, W_{\d}\cV). \]
\end{defn}

Note that  $\Phi_{v}$ is, in general, not invertible on $T_{\dr}({\sf M}_{\cO_{K_{v}}})$ and this $F_{v}$ is the Frobenius of  \cite[\S 4.1]{ABV} divided by $p_v$ (see \cite[\S 4.3 (4.b)]{ABV}), \ie $p_vF_{v}$ is the map associated to the Verschiebung ${\sf M}_{\cO_{K_{v}}}^{(p_v)}\to {\sf M}_{\cO_{K_{v}}}$.
 This choice is made so that the weight filtration on 1-motives and on the isocrystals are compatible.  We have:
\begin{lemma}\label{Tog}
$T_{\og}({\sf M}_{K})\in\mathbf{FOg}(K)$.
\end{lemma}
\begin{proof} In fact  $W_{0}\cV/W_{-1}\cV= T_{\og}([{\sf L}_{K}\to 0])$ so that the underlying $K$-vector space is $T_\dr ([{\sf L}_{K}\to 0])={\sf L}_{K}\otimes K$ and similarly for $W_{-2}\cV=T_{\og}([0\to {\sf T}_{K}])$  with underlying vector space $Y_{K}\otimes K$ where $Y_{K}$ the cocharacter group of the torus ${\sf T}_{K}$.   For  ${\sf T}$  the  $S$-torus which is a model of ${\sf T}_{K}$ let $Y$ be the group of cocharacters of $\sf T$. 
According to \cite[\S4.1]{ABV} for almost all unramified places $v$ the $\sigma_{v}$-semilinear homomorphism $F_{v}$ on $T_\cris([{\sf L}_{k_{v}}\to 0])={\sf L}\otimes W(k_{v})$ (respectively on $T_\cris([0\to  {\sf T}_{k_{v}}]))=Y\otimes W(k_{v})$) is $1\otimes \sigma_{v}$ (respectively the map $1 \mapsto  1\otimes p_{v}^{-1} \sigma_{v}$). 
Hence  $W_{0}\cV/W_{-1}\cV$ and $W_{-2}\cV$  are pure of weight $0$ and $-2$
respectively (in the sense of \S \ref{sec:mixedOgus}).

On the other hand $W_{-1}\cV/W_{-2}\cV= T_{\og}({\sf A}_{K}) $ and      $\sf A$ is a model over $S$ of the abelian variety  ${\sf A}_{K}$. Thanks to \cite[Thm.  A, p. 111]{ABV} for every unramified place $v$ of good reduction for ${\sf A}_{K}$ we can identify $T_\cris([0\to {\sf A}_{k_{v}}])$ and the $\sigma_{v}$-semilinear homomorphism $F_{v}$ with the covariant Dieudonn\'e module or equivalently with the crystalline homology of
the reduction ${\sf A}_{k_{v}}$ of ${\sf A}$. The latter defines a pure $F$-$K_{v}$-isocrystal of weight $-1$ thanks to a homological version of \cite[Cor. 1 (2)]{KM}.
\end{proof}

Recall that a functor between abelian categories with a weight filtration respects the splittings (in the sense of \cite[Def. D.2.2]{BVK}) if it takes pure objects to pure objects of the same weight.
\begin{propose} \label{OGreal}
There is a functor
\[T_{\og}\colon \MQ \to \mathbf{FOg}(K)\]
which associates  to a  $1$-motive ${\sf M}_{K}$  the object  $T_{\og}({\sf M}_{K})$ in $\mathbf{FOg}(K)$ provided by  Lemma \ref{Tog}. This functor respects the splittings and its essential image is contained in $\mathbf{FOg}(K)_{(1)}$ (see Definition \ref{fog1}).\end{propose}
\begin{proof}  It follows from the proof of the Lemma \ref{Tog}  and Remark \ref{r.leff}~(b). In fact, $T_{\og}({\sf M}_{K})$ is l-effective and e-effective in weight $0$ and Artin-Lefschetz in weight $-2$. Moreover, it is e-effective in weight $-1$ by \cite[Cor. 1 (2)]{KM}. Let $L \df T_{\dr}({\sf M})$ be the Lie algebra of the universal extension of the model ${\sf M}$ over $\Spec (\cO_K [1/n])$ as in Lemma \ref{1mod}. It is a $\cO_K [1/n]$-lattice in $V = T_{\dr}({\sf M}_K)$ and $L\otimes \cO_{K_v}\cong T_{\dr}({\sf M}_{\cO_{K_v}})$ is preserved by $p_vF_v$ as remarked above. Hence  $T_{\og}({\sf M}_{K})(-1)\in \mathbf{FOg}(K)^{\eff}$.
\end{proof}

\begin{defn}\label{bog-og}
The Bost-Ogus realization
\[T_{\bog}\colon \MQ \to \mathbf{BOg}(K)\]
associates  to a  $1$-motive ${\sf M}_{K}$  an object  $T_{\bog}({\sf M}_{K})\df \Psi T_{\og}({\sf M}_{K})(-1)$ in  $\mathbf{BOg}(K)$ where $\Psi$ is the functor defined in \eqref{psifun}. 
\end{defn} 

Note that if  ${\sf M}_{K}=[{\sf L}_{K}\to {\sf G}_{K} ]$, then  $T_{\bog}({\sf M}_{K})=\bLie( {\sf G}_{K}^\natural)$ with $\bLie$ as in  \eqref{bf}.
\subsection{The main Theorem} For a $1$-motive ${\sf M}_{K}=[{\sf u}_{K}\colon {\sf L}_{K}\to {\sf G}_{K}]$ over $K$ consider the universal extension ${\sf M}^\natural$ of a model ${\sf M} $ over $S=\Spec (\cO_K[1/n])$ and the morphism $\tau$ in \eqref{dia.ue}. For every unramified place $v$ in $\cP_{n}$ let ${\sf M}_{\cO_{K_{v}}}^\natural =[{\sf L}_{\cO_{K_{v}}}\to {\sf G}_{\cO_{K_{v}}}^\natural]$ be the base change of ${\sf M}^\natural$ to $\cO_{K_{v}}$. Recall the  $K_{v}$-linear map 
$$\delta_{v}\colon \Lie({\sf L}_{\cO_{K_{v}}}\otimes\G_{a})\otimes_{\cO_{K_{v}}}K_v \to  \Lie({\sf G}^\natural_{\cO_{K_{v}}})\otimes_{\cO_{K_{v}}}K_v$$
considered in Remark \ref{r.tau}.
\begin{lemma}\label{section} The homomorphism  $\delta_{v}$  is the unique section of $\Lie(\tau_{K_{v}})$ in the category of $F$-$K_{v}$-isocrystals.
\end{lemma}
\begin{proof} By Remark  \ref{r.tau} $\delta_v$ is a $K_v$-linear section of $\Lie(\tau_{K_{v}})$.   Since $p_{v}F_{v}$ on $T_\dr({\sf M}_{K_{v}})$ is the $K_{v}$-linear extension of the  Frobenius $\mathfrak f_{\cO_{K_{v}}}$ on $T_\cris({\sf M}_{k_{v}})$  discussed in \cite[\S 4.1]{ABV}, it suffices to show that $\delta_{v}$ is Frobenius equivariant  and, hence, provides a splitting of  $F$-$K_{v}$-isocrystals. Now Lemma \ref{wn}(iii) and the fact that morphisms of pure $F$-$K_{v}$-isocrystals of different weight are trivial imply that such a splitting is unique.

Recall that   $\mathfrak f_{\cO_{K_{v}}}$ on $T_{\cris}({\sf M}_{k_{v}})$ is defined by the Verschiebung $ {\sf M}_{k_{v}}^{(p)}\to {\sf M}_{k_{v}}$ thanks to the crystalline nature of universal extensions (see \cite[\S 4]{ABV}, \cite[Thm.  2.1]{ABe}). 
More precisely, for any $n\in \N$ there is a canonical  lift $V_n^\natural\colon \sigma_{v}^* {\sf M}_{W_n(k_{v})}^\natural\to {\sf M}_{W_n(k_{v})}^\natural$ of the Verschiebung on ${\sf M}_{k_{v}}^\natural $  where   $\sigma_{v}$ also denotes the Frobenius on $\spec (W_n(k_{v}))$, and the homomorphism $\Lie(V_n^\natural)\colon T_{\dr}({\sf M}_{W_n(k_{v})})\otimes_{\sigma_{v}} W_n(k_{v})\to T_{\dr}({\sf M}_{W_n(k_{v})})$ defines the Frobenius $\mathfrak f_{W_n(k_{v})}$ on $T_{\dr}({\sf M}_{W_n(k_{v})})$. Now, the construction is compatible with the truncation maps. Hence the  morphisms $V_n^\natural$, $n\geq 1$, provide a morphism $V^\natural$ of complexes of  formal schemes over $\cO_{K_v}=W(k_v)$
\begin{equation*}
\xymatrix{\sigma_{v}^*{\sf L}_{\cO_{K_{v}}}\ar[rr]^{(V^\natural)_{-1}}\ar[d]^{\sigma_{v}^* \hat{\sf u}}& &{\sf L}_{\cO_{K_{v}}}\ar[d]^{ \hat{\sf u}}\\
\sigma_{v}^*\widehat{{\sf G}}^\natural_{\cO_{K_{v}}}\ar[rr]^{(V^\natural)_{0}}& &\widehat {{\sf G}}^\natural_{\cO_{K_{v}}}.}
\end{equation*} Note that since the Verschiebung on ${\sf L}_{k_{v}}$ can be identified with the multiplication by $p_{v}$, the morphism $(V^\natural)_{-1} $ maps $x\in {\sf L}_{\cO_{K_{v}}}\times_{\cO_{K_{v}},\sigma_{v}} \mathrm{Spf} (\cO_{K_{v}})$ to $p_{v}x\in {\sf L}_{\cO_{K_{v}}}$. 
By the functoriality of (the formal) $\delta_{v}$ in Remark \ref{r.for} we then get a commutative diagram 
\begin{equation*}
\xymatrix{\Lie({\sf L}_{\cO_{K_{v}}}\otimes \G_a)\otimes_{\cO_{K_{v}}} K_{v}\otimes_{\sigma_{v}} K_{v}\ar[rrr]^{x\otimes1\otimes 1\mapsto p_{v}x\otimes 1  }\ar[d]^{\delta_{v}\otimes \mathrm{id}}&& &\Lie({\sf L}_{\cO_{K_{v}}}\otimes \G_a)\otimes_{\cO_{K_{v}}} K_{v} \ar[d]^{ \delta_{v}}\\
\Lie({\sf G}^\natural_{\cO_{K_{v}}})\otimes_{\cO_{K_{v}}} K_{v}\otimes_{\sigma_{v}} K_{v} \ar[rrr]^{ }&& & \Lie({\sf G}^\natural_{\cO_{K_{v}}})\otimes_{\cO_{K_{v}}} K_{v}. }
\end{equation*}
Since the upper (respectively, lower) horizontal arrow defines the Frobenius $\mathfrak f_{\cO_{K_{v}}}$ on $T_{\dr}([{\sf L}_{K_{v}}\to 0])$  (respectively, on $T_{\dr}( {\sf M}_{K_{v}})$), the result follows.
\end{proof}

\begin{lemma}\label{pro.tog0}
 The functor $T_{\og}$ is faithful.
\end{lemma}
\begin{proof}
Let $\varphi\colon {\sf M}_{K}\to {\sf N}_{K}$ be a morphism of $1$-motives such that $T_{\og}(\varphi)=0$.
In particular $T_{\dr}(\varphi)=0$. Then $n\varphi=0$ for a suitable $n$. Indeed,
 let $\varphi_{\C}$ be the base change of $\varphi$ to $\C$. Then $T_{\dr}(\varphi_\C)=0$ implies $ T_\C(\varphi_\C)\colon T_\Z({\sf M}_\C)\otimes_\Z \C\to T_\Z({\sf N}_\C)\otimes_\Z \C $ is the zero map by \cite[10.1.8]{De}. Hence $nT_\Z(\varphi_\C)=0$ for a suitable $n$. Then one concludes that $n\varphi_\C=0$ by \cite[10.1.3]{De} and hence $n\varphi=0$.
\end{proof}

\begin{remark} One could provide an alternative proof of Lemma \ref{pro.tog0} using an argument similar to the one adopted for the faithfulness in Theorem \ref{bost}.
\end{remark}

\begin{lemma}\label{pro.tog1}
 The functor $T_{\og}$ restricted to $\mathcal{M}_{0,\Q}$ is  full.
\end{lemma}
\begin{proof}
First  consider two $1$-motives ${\sf M}_{K}=[\Z^r\to 0]$, ${\sf N}_{K}=[\Z^s\to 0]$. Write $e_{1},\dots,e_r$ for the  standard basis of $\Z^r$ and similarly for $\Z^s$. We use the same letters for the induced bases on the de Rham realizations. Any morphism $\psi\colon T_{\dr}({\sf M}_{K})\to T_{\dr}({\sf N}_{K})$ corresponds  to an $s\times r$ matrix $C=(c_{ij}) \in M_{s,r}( K)$, \ie $\psi(e_{j})=  \sum_{i} c_{ij} e_{i}$. Observe that $c_{ij}\in \cO[1/n]$ for $n$  sufficiently divisible and hence $C\in M_{s,r}(\cO_{K_{v}})$ for  $v\in \cP_{n}$. If $\psi$ is a morphism in $\mathbf{Og}(K)$, the compatibility with the $F_{v}$'s implies that $\sigma_{v}(C)=C$ where $\sigma_{v}(C)=(\sigma_{v} (c_{ij}))$. Indeed \[F_{v}(\psi(e_{j}))=F_{v}(\sum_{i} c_{ij}e_{i})=\sum_{i}\sigma_{v}(c_{ij})e_{i}, \] and \[  \psi(F_{v}(e_{j}))=\psi(e_{j})=\sum_{i} c_{ij} e_{i}.\] Now $\sigma_{v}(c_{ij})=c_{ij}$ for  all $v\in \cP_{n}$,  is equivalent to  $c_{ij}^{p_{v}}\equiv c_{ij} \text{ (mod }p_{v})$  for  all $v\in \cP_{n}$. We conclude by Kronecker's theorem that $c_{ij}\in \Q$; hence $C\in  M_{s,r}(\Q)$. Let $m$ be the a positive integer such that $mC\in   M_{s,r}(\Z)$. Then  $m\psi=\Lie(\varphi\otimes \mathrm{id})$ with $\varphi\in \mathrm{Hom}({\sf M}_{K},{\sf N}_{K}) \simeq  M_{s,r}(\Z)  $ the homomorphism which maps $e_{j}$ to $\sum_{i} mc_{ij}e_{i}$  in degree $-1$.

We conclude the proof of the fullness by Galois descent. Let ${\sf M}_{K}=[{\sf L}_{K}\to 0], {\sf N}_{K}=[{\sf F}_{K}\to 0]$ be $1$-motives in $\mathcal{M}_{0,\Q}$ and let $\psi\colon T_{\dr}({\sf M}_{K})\to T_{\dr}({\sf N}_{K})$ be a morphism in $\mathbf{FOg}(K)$; in particular the base change of $\psi$ to $K_{v}$ is compatible with the $F_{v}$'s for almost all places $v$. 
Let $K^\prime$ be a finite  Galois  extension of $K$  such that ${\sf L}_{K^{\prime}}$ and ${\sf F}_{K^{\prime}}$ are split. Then, any unramified place $v^{\prime}$ of $K^{\prime}$ is above an unramified place $v$ of $K$ and   $K^{\prime}_{v^{\prime}}/K_{v}$ is a finite unramified extension. Then by \cite[Corollary 4.2.1]{ABV} and the fact that the de Rham realization and the Verschiebung morphism are compatible with extension of the base, the formation of $T_{\og}({\sf M}_{K})$ behaves well with respect to base field extension. Let $\psi_{K^{\prime}}\colon T_{\dr}({\sf M}_{K^{\prime}})\to T_{\dr}({\sf N}_{K^{\prime}})$ be the morphisms in $\mathbf{FOg}(K^{\prime})$ induced by $\psi$. 
Denote by the same letter also the associated morphism of vector groups ${\sf L}_{K^{\prime}}\otimes \G_{a}\to {\sf F}_{K^{\prime}}\otimes \G_{a}$.
By above discussions we may assume that $\psi_{K^{\prime}}$ comes from a morphism $\varphi_{K^{\prime}}\colon  {\sf L}_{K^{\prime}}\to  {\sf F}_{K^{\prime}}$. Let us denote by $\zeta$ both an element in $\mathrm{Gal}(K^{\prime}/K)$ and the induced morphism on $1$-motives. In order to check that $\varphi_{K^{\prime}}$ descends to $K$, it suffices to check that $\zeta  \circ \varphi_{K^{\prime}}=\varphi_{K^{\prime}}\circ \zeta$. 
Since we work over fields of characteristic $0$, the morphism $\iota_{{\sf F}}\colon {\sf F}_{K^{\prime}}\to  {\sf F}_{K^{\prime}}\otimes \G_{a}, x\mapsto x\otimes 1$ has trivial kernel. Hence it is sufficient to check that  $\iota_{{\sf F}}\circ \zeta \circ \varphi_{K^{\prime}}=\iota_{{\sf F}}\circ\varphi_{K^{\prime}}\circ \zeta$. Now,
\[\iota_{{\sf F}}\circ \zeta \circ \varphi_{K^{\prime}}=
\zeta \circ \iota_{{\sf F}} \circ \varphi_{K^{\prime}}=
\zeta \circ \psi_{K^{\prime}}  \circ \iota_{{\sf L}}=
 \psi_{K^{\prime}}  \circ \zeta\circ \iota_{{\sf L}}=
  \psi_{K^{\prime}}  \circ \iota_{{\sf L}}\circ \zeta=
  \iota_{{\sf F}}\circ\varphi_{K^{\prime}}\circ \zeta.
\]
This concludes the proof.
\end{proof}

\begin{thm}\label{thm.tog}
The functor $T_{\og}\colon \MQ \to \mathbf{FOg}(K)$  is fully faithful.
\end{thm}
\begin{proof} The faithfulness was proved in Lemma \ref{pro.tog0}.
For the fullness, let ${\sf M}_{K}=[{\sf u}_{K}\colon {\sf L}_{K}\to {\sf G}_{K}]$ and ${\sf N}_{K}=[{\sf v}_{K}\colon {\sf F}_{K}\to {\sf H}_{K}]$ be $1$-motives.  Suppose  given a morphism $\psi\colon T_{\og}({\sf M}_{K})\to T_{\og}({\sf N}_{K})$ in $\mathbf{FOg}(K)$. By Definition \ref{bog-og} we also get a morphism $\psi^{\prime}\colon T_{\bog}({\sf M}_{K})\to T_{\bog}({\sf N}_{K})$ in  $\mathbf{BOg}(K)$. Using Theorem \ref{bost} the morphism $\psi^{\prime}$ comes  from a morphism $\tilde g_{K}\colon {\sf G}_{K}^{\natural}\to {\sf H}_{K}^{\natural}$, \ie  $\bLie (\tilde g_{K}) = m\psi^{\prime}$ in $\mathbf{BOg}(K)$ for a suitable $m\in \N$. We may assume $m=1$. By Chevalley theorem (see \cite[Lemma 2.3]{Co}) $\tilde g_{K}$ yields a morphism on the semi-abelian quotients $g_{K}\colon {\sf G}_{K}\to {\sf H}_{K}$. Now consider the morphism
$T_{\og}(g_{K})\colon T_{\og}({\sf G}_{K})\to T_{\og}({\sf H}_{K})$ and compare with the morphism induced by $\psi$ on weight $-1$ parts as displayed in the following commutative diagram
\[
\xymatrix{
 0\ar[r]&T_{\og}({\sf G}_{K})\ar[r]\ar[d]^{\psi_{-1}} &  T_{\og}({\sf M}_{K}) \ar[r]  \ar[d]^{\psi} & T_{\og}({\sf L}_{K}[1])\ar[r]\ar[d]^{\psi_{0}} &0\\
0\ar[r]&T_{\og}({\sf H}_{K})  \ar[r]& T_{\og}({\sf N}_{K}) \ar[r] & T_{\og}({\sf F}_{K}[1])\ar[r] &0}
 \]
Since by construction $T_{\bog} (g_{K}) = \psi^{\prime}_{-1}$ in $\mathbf{BOg}(K)$ we deduce that $T_{\og} (g_{K}) = \psi_{-1}$ as well. In fact, $T_{\dr} (g_{K})= T_{\bog} (g_{K})= T_{\og} (g_{K})$ coincide on the underlying $K$-vector spaces via $T_{\dr}$.

It follows from the Lemma \ref{pro.tog1} that there exists a morphism $f_{K}\colon {\sf L}_{K}\to {\sf F}_{K}$ such that  $ T_{\og} (f_{K}) = m\psi_{0}\colon T_{\og}({\sf L}_{K}[1])\to T_{\og}({\sf F}_{K}[1])$ for an $m\in \N$. As above, we may assume $m=1$.

Note that if the pair $(f_K,g_K)$ gives a morphism of $1$-motives $\varphi_{K}\colon {\sf M}_{K}\to {\sf N}_{K}$, i.e., if  $  g_{K}\circ {\sf u}_{K}={\sf v}_{K}\circ f_{K}$, then $T_{\bog}(\varphi_{K})=\psi^{\prime}$ by construction. As the functor $\Psi$ of \eqref{psifun} is faithful, the fact that $T_{\og}(\varphi_{K})-\psi$ induces the zero morphism between the Bost-Ogus realizations implies that $T_{\og}(\varphi_{K})=\psi$ in $\mathbf{FOg}(K)$.
  
It is then sufficient to check that, up to multiplication by a positive integer, we have $\tilde g_{K}\circ {\sf u}^{\natural}_{K}={\sf v}^{\natural}_{K}\circ f_{K}$. 
After replacing $K$ with a finite extension we may further assume that ${\sf L}$ and ${\sf F}$ are constant. Take  $v$ an unramified place of good reduction both for ${\sf M}_{K}$ and ${\sf N}_{K}$ such that $\tilde g_{K}$ extends to a morphism $\tilde g  \colon {\sf G}^{\natural}\to {\sf H}^{\natural}$ over $W(k_{v})=\cO_{K_{v}}$.  
In particular, ${\sf u}^{\natural}_{K}$ and ${\sf v}^{\natural}_{K}$ extend to morphisms ${\sf u}^{\natural} \colon {\sf L}\to {\sf G}^{\natural}, {\sf v}^{\natural}\colon {\sf F}\to {\sf H}^{\natural}$ over $\cO_{K_{v}}$ and hence are determined by the induced homomorphisms between the $\cO_{K_{v}}$-rational points. It suffices then to prove that the following diagram
\begin{equation}\label{lgfh}
\xymatrix{
 {\sf L}(\cO_{K_{v}})\ar@{->}[r]^{{\sf u}^\natural\hspace*{0.5cm}}\ar[d]^f & {\sf G}^{\natural}(\cO_{K_{v}}) \ar[d]^{\tilde g}\\
{\sf F}(\cO_{K_{v}})\ar@{->}[r]^{{\sf v}^\natural\hspace*{0.5cm}} & {\sf H}^{\natural}(\cO_{K_{v}})}
\end{equation}
commutes, up to multiplication by a positive integer. Consider the following diagram 
\[
\xymatrix{
 {\sf L}(\cO_{K_{v}}) \ar[r]^(0.4){\alpha_{\sf L}}\ar[d]^f &\Lie ({\sf L}\otimes \G_{a})\! \otimes\!  K_{v}\ar[r]^{\delta_{v}^{\sf M}}\ar[d]^{\psi_{0}\otimes K_{v}} &  \Lie ({\sf G}^{\natural})\! \otimes\!  K_{v}  \ar[r]^(0.4){\log_{  {\sf G}^{\natural},\Q}^{-1}}_{\simeq}  \ar[d]^{\psi\otimes K_{v}} & {\sf G}^{\natural}(\cO_{K_{v}})\! \otimes_{\Z }\! \Q   \ar[d]^{\tilde g\otimes \Q}\\
{\sf F}(\cO_{K_{v}}) \ar[r]^(0.45){\alpha_{\sf F}}&\Lie ({\sf F}\otimes \G_{a})\! \otimes\!  K_{v}\ar[r]^{\delta_{v}^{{\sf N}}} &  \Lie ({\sf H}^{\natural})\! \otimes \! K_{v}  \ar[r]^(0.45){\log_{ {\sf H}^\natural,\Q}^{-1}}_{\simeq} & {\sf H}^{\natural}(\cO_{K_{v}})\! \otimes_{\Z }\! \Q }
 \] 
where the map $\alpha_{\sf L}$ is the composition of the canonical map $ {\sf L}(\cO_{K_{v}}) \to (L\otimes \G_{a})(\cO_{K_{v}})\otimes_\Z \Q, x\mapsto (x\otimes 1)\otimes 1$ with $ \log_{L\otimes \G_{a},\Q}$ and similarly for $\alpha_{\sf F}$. Note that we here identify
$T_{\dr} ({\sf L}_{K}[1])_{v}$ with $\Lie ({\sf L}\otimes \G_{a})\otimes  K_{v}$ and 
$T_{\dr} ({\sf F}_{K}[1])_{v}$  with $\Lie ({\sf F}\otimes \G_{a})\otimes  K_{v}$; we also identify   $T_{\dr}({\sf M}_{K})_{v}$ with $\Lie ({\sf G}^{\natural})\otimes K_{v}$ and $T_{\dr}({\sf N}_{K})_{v}$ with $\Lie ({\sf H}^{\natural})\otimes K_{v}$.
Hence the most left square commutes since by definition $T_\og (f_K)= \psi_0$. The commutativity of the square in the middle follows from Lemma \ref{section} as  $\psi$ and $\psi_0$ are morphisms in $\mathbf{Og}(K)$ so that $\psi\otimes K_v$ and $\psi_0\otimes K_v$ commutes with $F_v$.
The last square on the right commutes by functoriality of the logarithm as $\psi\otimes K_{v} = \psi^{\prime}\otimes K_{v} = \Lie  (\tilde g_{K} \otimes K_{v})$ on the underlying $K_v$-vector spaces.  Finally, note that the composition of the upper (respectively, lower) horizontal arrows is ${\sf u}^\natural\otimes \mathrm{id}_{\Q}$ (respectively, ${\sf v}^\natural \otimes \mathrm{id}_{\Q}$) by definition of $\delta_{v}$ in \eqref{eq.del}. Hence \eqref{lgfh} commutes up to multiplication by a positive integer.  
\end{proof}

\begin{example}\label{ex.fund}
Let ${\sf M}_{K}=[{\sf u}_{K}\colon \Z\to \G_{m,K}]$ and ${\sf N}_{K}=[{\sf v}_{K}\colon \Z\to \G_{m,K}]$ be two $1$-motives over $K$. Set $a\df {\sf u}_{K}(1)\in K^\ast$ and $b\df {\sf v}_{K}(1)\in K^\ast$. Note that any morphism 
$(f_{K},g_{K})\colon {\sf M}_{K}\to {\sf N}_{K}$ is of the type $f_{K}=m$, $g_{K}=r$ with $a^r=b^m$.
In particular, for general $a,b\in K^\ast$ the  unique morphism between $ {\sf M}_{K}$ and ${\sf N}_{K}$ is the zero morphism.  
By Definition \ref{Tdr}, it follows from \eqref{dia.ue} that \[T_{\dr}({\sf M}_{K})=T_{\dr}({\sf N}_{K})= \Lie(\G_{m,K})\oplus \Lie(\G_{a,K})=K\oplus K.\]
 The de Rham realisation of the two $1$-motives yields objects in $\mathbf{Og}(K)$ with the same underlying structure of filtered $K$-vector spaces: the filtration being induced by the weight filtration $W_{-1}{\sf M}_{K}= [0 \to \G_{m,K}]$ and $\gr_0^W {\sf M}_{K} =[\Z \to 0]$.
 As described in Definition \ref{TOg} (\cf the proof of Lemma \ref{Tog}),  for any unramified place $v$ of $K$, $T_{\dr}({\sf M}_{K_{v}})=K_{v}\oplus K_{v}$ is endowed with the $\sigma_{v}$-semilinear operator $F_{v}$  such that for $(x, y)\in K_{v}\oplus K_{v}$ we have that   $F_{v}(x, 0)=(p_{v}^{-1}\sigma_{v}(x),0)$ and $\bar{F_{v}((x,y))}= \sigma_{v}(y)$ where the $\bar{(\ \  )}$ stands for the class in the weight $0$ quotient $K_{v}$; the same holds for ${\sf N}_{K}$. 

Note that, in general, we have several $K$-linear homomorphisms $T_{\dr}({\sf M}_{K})\to T_{\dr}({\sf N}_{K})$ that preserve the filtration and commute with the $F_{v}$'s on the graded pieces of the filtration but do not arise from morphisms of $1$-motives (even up to isogeny).  For example, take  $a=1$ and $b =2$ and the identity map on $K \oplus K$. This example shows that knowing fullness for the pure weight parts is not enough to deduce the Theorem \ref{thm.tog}. 

For general $a,b\in K^\ast$, knowing $F_{v}$ on $T_{\dr}({\sf M}_{K})\otimes K_{v} =K_{v}\oplus K_{v}$ is equivalent to give a Frobenius equivariant splitting $\delta_{v}^{\sf M}\colon K_{v}\to K_{v}\oplus K_{v}$ of the weight $0$ quotient $K_{v}$ thanks to Lemma \ref{section}. Now, assume $a,b\in \cO_{K_{v}}$. Then, by Examples \ref{ex.ga}, we can write $\delta_{v}^{\sf M}(1)=(
\log(a^{p_{v}(p_{v}^{n_{v}}-1}))/p_{v}(p_{v}^{n_{v}}-1)\oplus 1$. 
As a consequence, the identity map on $K\oplus K$ commutes with the sections $\delta_{v}^{\sf M},\delta_{v}^{\sf N} $  if and only if  $a^{p_v(p^{n_{v}}_{v}-1)}=b^{p_v(p_{v}^{n_{v}}-1)}$, thus  if, and only if, $a^{ p_{v}^{n_{v}}-1}=b^{p_{v}^{n_{v}}-1} $ (as $K_v$ does not contain non-trivial $p_v$-roots of unity being absolutely unramified).
 We conclude that there exists a  morphism $T_\og({\sf M}_{K})\to T_\og({\sf N}_{K})$ in $\mathbf{FOg}(K)$ which is identity on the underlying vector spaces if and only if  ${\sf M}_{K}={\sf N}_{K}$.
\end{example}

\begin{remark} \label{psibog}
If we work with $\mathbf{BOg}(K)$ and even with the filtered analogue, it is not true that the functor $T_{\bog}$ is full on 1-motives, in general. For example, take ${\sf M}= \Z[1]$ and note that $\End_{\MQ}({\sf M}) = \Q$ while $\End_{\mathbf{BOg}(K)}(T_{\bog} ({\sf M}))= K$.
\end{remark}


\vfill


\newpage

\begin{thebibliography}{}

\bibitem{An} Y. Andr\'e: {\it Une introduction aux motifs (motifs purs, motifs mixtes, p\'eriodes)} Panoramas et Synth\`eses {\bf 17}, Soci\'et\'e Math\'ematique de France, 2004.

\bibitem{ABV} F. Andreatta \& L. Barbieri-Viale: Crystalline realizations of 1-motives, {\it Math. Ann.} {\bf  331} N. 1 (2005) 111-172.

\bibitem{ABe}  F. Andreatta \& A. Bertapelle: Universal extension crystals of 1-motives and applications, {\it Journal of Pure and Appl. Algebra} {\bf 215} (2011) 1919-1944.

\bibitem{BVB} L. Barbieri-Viale \&  A. Bertapelle: Sharp de Rham realization, {\it Adv. Math.} {\bf 222} (2009) 1308-1338.

\bibitem{Ber} A. Bertapelle: Deligne's duality for de Rham realizations of $1$-motives,  {\it Math. Nachr.} {\bf 282} (2009) 1637-1655.

\bibitem{BVK}   L. Barbieri-Viale \& B. Kahn: {\it On the derived category of $1$-motives} Ast\'erisque {\bf 381} Soci\'et\'e Math\'ematique de France, Paris, 2016.

\bibitem{BBM} P. Berthelot,  L. Breen \&  W. Messing: {\it Th\'eorie de Dieudonn\'e cristalline II} Lect. Notes in Math. {\bf 930}, Springer, New York, 1982.

 
\bibitem{BLR} S. Bosch, W. L\"utkebohmert \& M. Raynaud: {\it N\'eron models} Erg. der Math. Grenz. {\bf{21}}, Springer-Verlag,  Berlin, 1990.


\bibitem{Bo} J.-B. Bost:  Algebraic leaves of algebraic foliations over number fields {\it Publ. Math.} IHES  {\bf 93} (2001) 161-221.
 
\bibitem{C-L} A. Chambert-Loir: Th\'eor\`emes d'alg\'ebricit\'e en g\'eom\'etrie diophantienne d'apr\`es J.-B. Bost, Y. Andr\'e, D. \& G. Chudnovsky, S\'eminaire Bourbaki Mars 2001, n. 886.

\bibitem{Ch} B. Chiarellotto: Weights in rigid cohomology applications to unipotent $F$-isocrystals, {\it Annales Scientifiques} ENS {\bf 31} Issue 5 (1998) 683-715.

\bibitem{Co} B. Conrad: A modern proof of Chevalley's theorem on algebraic groups, {\it J. Ramanujam Math. Soc.} {\bf 17} (2002) 1-18.

\bibitem{DG} M. Demazure \& P. Gabriel:  {\it Groupes
alg\'ebriques} Tome I: G\'eom\'etrie alg\'ebrique, g\'en\'eralit\'es, groupes commutatifs, Masson \& Cie, \'Editeur, Paris, 1970. 

\bibitem{sga3}   M. Demazure \& A. Grothendieck   (Eds.): {\it Sch\'emas en groupes} S\'eminaire de G\'eom\'etrie Alg\'ebrique du Bois Marie 1962-64 (SGA 3).   Lect. Notes in Math. {\bf 151}, Springer-Verlag, Berlin, 1970.

\bibitem{De} P. Deligne: Th\'eorie de Hodge III. {\it Publ. Math.} IHES {\bf 44} (1974) 5-78.

\bibitem{GM} B. Edixhoven, G. van der Geer \&  B. Moonen: {\it Abelian varieties}  Preprint: \url{http://www.math.ru.nl/~bmoonen/research.html}.

\bibitem{ega}  A. Grothendieck (with J. Dieudonn\'e): El\'ements de g\'eom\'etrie alg\'ebrique.   {\it Publ. Math.} IHES {\bf 11} (III, 1-5), {\bf 28} (IV, 8-15) and {\bf 32} (IV, 16-21) (1961-1967).

\bibitem{JA} U. Jannsen: Mixed motives, motivic cohomology and Ext-groups, in the {\it Proceedings
of the International Congress of Mathematicians} Vol. 1, 2 (Zurich, 1994), 667-679, Birkh\"auser, Basel, 1995.

\bibitem{KM} N. Katz \& W. Messing: Some consequences of the Riemann hypothesis for varieties over finite fields, {\it Invent. Math.} {\bf 23} (1974), 73-77.

\bibitem{KW} M. Kisin \& S. Wortmann: A note on Artin motives, {\it Math. Research Letters} {\bf 10} (2003) 375-389.

\bibitem{Me}  W. Messing: {\it The crystals associated to Barsotti-Tate groups: with applications to abelian schemes} Lect. Notes in Math. {\bf 264}, Springer-Verlag, Berlin, Heidelberg, New York, 1972.

\bibitem{Mu} D. Mumford: {\it Abelian varieties} Tata Institute of Fundamental Research Studies in Mathematics, 5, Bombay, 1970.

\bibitem{Ta}  J. T. Tate: p-divisible groups, in  Proc. Conf. Local Fields (Driebergen, 1966) 158-183 Springer, Berlin 1967. 

\end{thebibliography}
\end{document}